\documentclass[11pt,a4paper]{article}

\usepackage{amsmath,amsfonts,amssymb,amsthm,bm,eqlist,tabularx,tabu,threeparttable,longtable,exscale,enumitem}
\numberwithin{equation}{section}
\usepackage[mathscr]{eucal}
\usepackage[round]{natbib}
\usepackage[dvipsnames]{color} 
\usepackage[Lenny]{fncychap} 
\usepackage{multirow,graphicx,algpseudocode,framed,booktabs,framed}
\usepackage[affil-it]{authblk}
\usepackage[labelfont=bf,format=plain,justification=raggedright,singlelineche ck=false]{caption}

\usepackage{tikz}
\usepackage[utf8]{inputenc} 
\usepackage[english]{babel} 
\usepackage[T1]{fontenc}    
\usepackage{braket} 
\usepackage[ruled, vlined, boxed,commentsnumbered,linesnumbered]{algorithm2e}
\DeclareMathOperator*{\argmax}{arg\,max}
\DeclareMathOperator*{\argmin}{arg\,min}
\usepackage[colorlinks=true, linkcolor=blue,
anchorcolor=blue, citecolor=blue, urlcolor=blue]{hyperref}

\hypersetup{
	colorlinks=true,
	linkcolor=blue,
	citecolor=blue,
	citebordercolor={1 1 0},
}

\newtheorem{theorem}{Theorem}[section]
\newtheorem{lemma}{Lemma}
\newtheorem{proposition}[theorem]{Proposition}

\newtheorem{remark}[theorem]{Remark}

\font\bigbf=cmbx10 scaled \magstep3

\title{\bigbf Mechanism design for coordinating vehicle-based mobile sensing tasks within the ride-hailing platform}

\author[1]{Shenglin Liu}
\author[1]{Qian Ge\thanks{Corresponding author. Email: geqian@swjtu.edu.cn}}
\author[2]{Ke Han\thanks{Corresponding author. Email: kehan@swjtu.edu.cn}}
\author[3]{Daisuke Fukuda}
\author[4]{Takao Dantsuji}
\affil[1]{School of Transportation and Logistics, Southwest Jiaotong University, Chengdu, China}
\affil[2]{School of Economics and Management, Southwest Jiaotong University, Chengdu, China}
\affil[3]{Department of Civil Engineering, The University of Tokyo, Tokyo, Japan}
\affil[4]{Institute of Transport Studies, Monash University, Australia}
\begin{document}
	\maketitle
	\begin{abstract}
		This paper evaluates the benefit of integrating vehicle-based mobile crowd-sensing tasks into the ride-hailing system through the collaboration between the data user and the ride-hailing platform. In such a system, the ride-hailing platform commissions high-valued sensing tasks to idle drivers who can undertake either ride-hailing or sensing requests. Considering the different service requirements and time windows between sensing and ride-hailing requests, we design a staggered operation strategy for ride-hailing order matching and the sensing task assignment. The auction-based mechanisms are employed to minimize costs while incentivizing driver participation in mobile sensing. To address the budget deficit problem of the primal VCG-based task assignment mechanism, we refine the driver selection approach and tailor the payment rule by imposing additional budget constraints. We demonstrate the benefits of our proposed mechanism through a series of numerical experiments using the NYC Taxi data. Experimental results reveal the potential of the mechanism for achieving high completion rates of sensing tasks at low social costs without degrading ride-hailing services. Furthermore, drivers who participate in both mobile sensing tasks and ride-hailing requests may gain higher income, but this advantage may diminish with an increasing number of such drivers and higher demand for ride-hailing services.
	\end{abstract}
	
	\section{Introduction}\label{section:intro}
	In the past decade, the ride-hailing industry emerged as a significant door-to-door transportation mode, reshaping urban mobility in many countries. However, the market's rapid growth appears to have slowed in its two largest markets, i.e., the United States \citep{ridehailus} and China \citep{ridehailchina}, leading the ride-hailing platforms into fierce competition with both external and internal companies. This motivates ride-hailing platforms to explore new niche markets, including other vehicle-based crowd-sourced activities, alongside human mobility \citep{li2014share}. For example, food or freight delivery service constitutes a non-negligible part of the revenue of Uber \citep{uber2023exp}. In China, the problem is more subtle with the advent of third-party integrators of the ride-hailing service, such as Amap and Baidu. Many car rental companies enter the ride-hailing business by leasing vehicles to drivers and assigning the ride-hailing orders to them with the help of the integrators. In other words, the ride-hailing market itself is ironically a new niche market for car rental companies. Instead, some researchers and practitioners have explored the integration of taxi or ride-hailing service with vehicle-based mobile sensing \citep{xu2019ilocus, wu2020application}.
	
	In the mobile sensing system, either human workers or vehicles equipped with \textit{ad hoc} sensors visit specified points of interest passively or in response to platform requests. They transmit collected data to processing systems and receive rewards from data users. This sensing approach has been applied to a variety of scenarios in urban management, including pollution monitoring \citep{hasenfratz2015deriving,jezdovic2021crowdsensing}, infrastructure maintenance \citep{eriksson2008pothole}, and traffic congestion management \citep{guo2022sensing}. However, vehicle-based mobile sensing, while potentially cost-effective for data users, is limited by sampling bias due to the sensitivity of the resulting dataset to the vehicles' mobility patterns. Despite the development of incentive mechanisms aimed at synchronizing vehicle trajectories with task requirements, they often rely on either compensation or punishment of drivers for their involvement in rarely visited tasks, deviating from the merits of crowd-sourcing. Meanwhile, a small and agile fleet of vehicles dedicated to mobile sensing is necessary to compensate for the dataset's insufficiencies.
	
	Inspired by \cite{fehn2023integrating}, we investigate an operational strategy for vehicle-based mobile crowd-sensing that integrates sensing tasks with ride-hailing requests, incentivizing drivers to undertake tasks while maintaining service levels for regular riders. In this business model, a data user collaborates with a ride-hailing platform by releasing a third-party app for sensing tasks and commissioning these tasks to a pool of available drivers. Both the driver and the ride-hailing platform receive commissions from the data user for successfully completed sensing tasks. Any remaining sensing tasks, such as areas where no driver is available, would be handled by a dedicated vehicle fleet owned by the data user. This approach allows the data user to cut down its own vehicle fleet for data sensing, enables drivers to earn extra income through side gigs, and may increase profits for the ride-hailing platform. The operational strategy is illustrated in Figure \ref{fig:enter-label}.

	\begin{figure}[htbp]
		\centering
		\includegraphics[width=\textwidth]{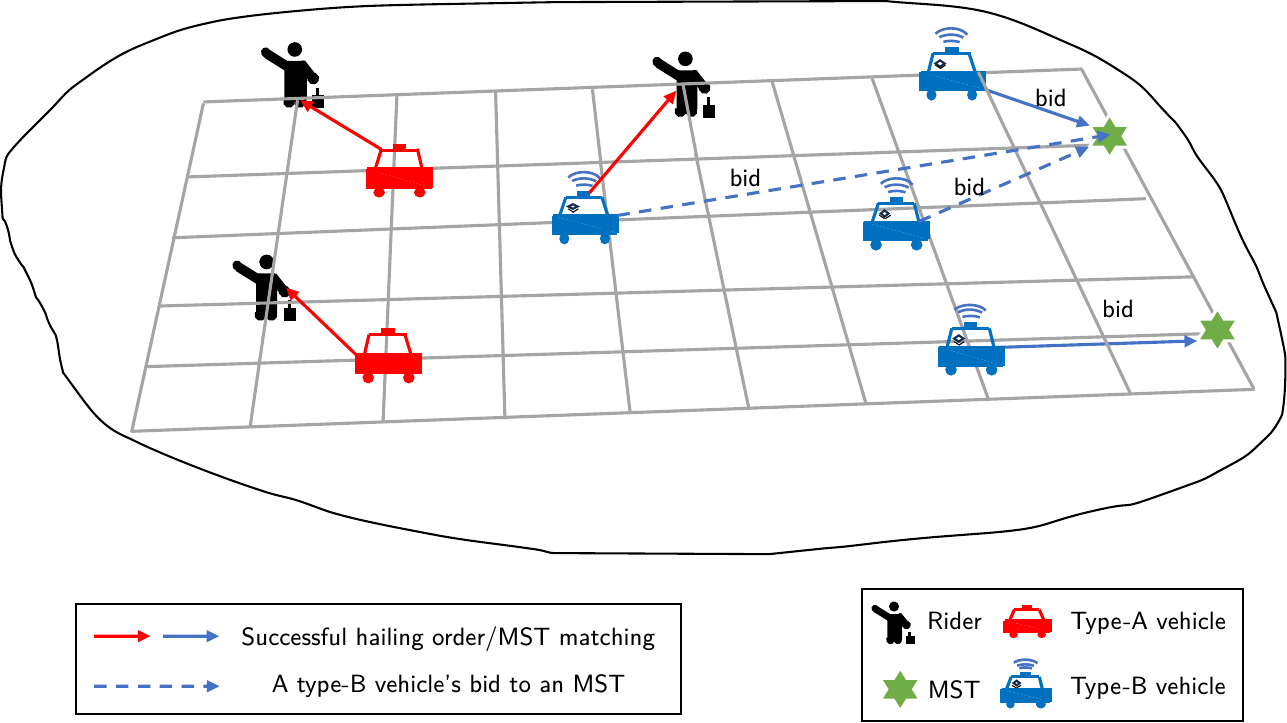}
		\caption{Schematic of the ride-hailing order and MST assignment system. Type-A vehicles are exclusively designated for ride-hailing orders, whereas Type-B vehicles have the flexibility to handle both passenger delivery and mobile sensing tasks. Initially, the ride-hailing platform matches vehicles with riders and then allocates mobile sensing tasks to the remaining Type-B vehicles through an auction-based procedure.}
		\label{fig:enter-label}
	\end{figure}
	
	Arguably, sensing tasks could be outsourced to ride-hailing platforms by generating an equivalent number of `artificial ride-hailing requests' at points of interest (POIs) and assigning these tasks to drivers through the same procedure used for ordinary ride pickups. While this strategy may benefit the data user, it may not align with the operational objectives of the ride-hailing platform, which must prioritize ride requests to ensure service levels and maintain customer loyalty. Additionally, sensing tasks differ from ride-hailing requests not only in their spatial distributions and pickup/drop-off time windows (or earliest commencement and latest completion times for sensing tasks), but also in factors such as service quality and driver commitment. The driver-rider matching pattern, along with the prescribed taxi fare for ordinary ride-hailing requests, may not be suitable for sensing tasks. In the former case, the ride-hailing platform would have to address rider complaints about dissatisfied service and possibly downgrade the ratings of responsible drivers. However, drivers could easily opt out of undesirable sensing tasks without facing significant penalties, as the mobile sensing tasks (MSTs) constitute only a small portion of all tasks. To tackle these challenges, the ride-hailing platform may need to implement different rules for matching drivers with MSTs and ride-hailing requests.
	
	To the best of our knowledge, this paper presents the first operational strategy for addressing the integrated mobile sensing and ride-hailing task assignment problem within a unified ride-hailing system, aimed at benefiting all involved stakeholders. Specifically, our objectives are as follows:
	\begin{enumerate}[label=(\roman*)]
		\item How can we ensure that all stakeholders, including the data user, the ride-hailing platform, and the drivers, derive benefits or, at the very least, do not suffer any negative impact from the introduction of MSTs to the ride-hailing platform?
		\item What methods can be employed to incentivize drivers to accept and successfully complete sensing tasks? 
		\item How to conduct sensing tasks while maintaining a satisfactory service level of trip matching?
		\item Given a limited budget from the data user, how can we achieve adequate coverage of areas of interest through vehicle-based mobile crowd-sensing?
		\item What is the impact of system design parameters on the effectiveness of the operational strategy?
	\end{enumerate}
	We make the following contributions in this paper:
	\begin{enumerate}[label=(\roman*)]
		\item  We introduce an operational strategy for ride-hailing platforms that coordinates the assignment of trip requests and mobile sensing tasks. This strategy involves three phased assignments: trip request matching, simultaneous trip request matching, and mobile-sensing task bidding, and mobile-sensing task allocation. To prioritize the level of service for ride-hailing users, trip requests are prioritized in the first and second phases, with the third phase dedicated to mobile-sensing tasks. Notably, only redundant vehicles failing to match trip requests are eligible to bid for mobile sensing tasks in the second phase. Additionally, a driver's bid is revoked upon being matched to a trip request.
		\item While formalizing trip request matching as a bipartite graph matching problem, we commission mobile-sensing tasks to drivers through an auction-based approach. This approach is a refinement of the classical Vickrey-Clarke-Groves (VCG) mechanism, with an added guarantee on the total budget. This refinement ensures favorable economic properties such as individual rationality (IR), incentive compatibility (IC), and budget balance (BB).
		\item We evaluate the performance of the proposed model framework on the TLC Trip Record Data. Computational results demonstrate that our coordinated operational strategy enhances the welfare of the data user, ride-hailing platform, and drivers without compromising service levels for ride users. Furthermore, our refined mechanism reduces the total cost of mobile sensing compared to the VCG mechanism. This cost reduction becomes more significant with the expansion of the vehicle fleet.
	\end{enumerate}
	
	The remainder of this paper is organized as follows. Section \ref{section:review} reviews literature related to this study. Section \ref{section:model} presents the model framework for coordinating the trip matching and the task assignment problems. Our refined mechanism for task assignment and its economic properties are described in Section \ref{section:rbc}. Section \ref{section:numer} summarizes the main results and the managerial insights of computational experiments. Section \ref{section:conc} concludes this work.
	
	\section{Literature review}\label{section:review}
	This paper explores the coordination strategy of trip matching and task assignment within the ride-hailing platform. We categorize the related literature into three main streams: (1) vehicle-based mobile crowdsensing, with an emphasis on the use of taxis, (2) coordination of ride-hailing services with delivery or other tasks, and (3) auction-based mechanisms for resource allocation in transportation.
	
	\subsection{Taxi-based mobile crowdsensing}
	Vehicles are widely used for mobile crowdsourced tasks, especially urban data collection, and crowdsourced delivery, for their advantages in high mobility, scattered distribution, and low cost. The mobile crowdsensing tasks could be performed either by \textit{active data collectors} who are primarily dedicated to the sensing tasks or \textit{passive data collectors} who record the data while conducting other tasks. There is a large body of studies addressing the data collection, completion, assimilation, and training problems of active data collectors for mobile sensing. One may refer to \citep{JI2023104874} for a recent review on this topic. 
	
	Our work focuses exclusively on ride-hailing taxis, which fall in the group of passive data collectors. In contrast to the active data collectors, the routes and trajectories of the taxis could not be planned or coordinated by a centralized operator, i.e., the sensing power is sensitive to the driving and cruising behavior of taxi drivers\citep{o2019quantifying}. Some unpopular streets or areas are rarely or even never visited by taxis. Nonetheless, providing information on the spatial and temporal demand to the drivers will affect their behaviors and increase social welfare \citep{zhang2020learning}. Further explorations on the mobility behavior of taxi drivers may refer to \cite{wang2019mobile}. 
	
	Instead of controlling the trajectories of taxis directly, for which the effect is doubted, a more realizable solution is to develop incentive mechanisms to motivate the drivers to cover the ideal areas \citep{JI2023104874}. \citet{masutani2015sensing} develop a centralized decision framework that recommends routes to a subset of vehicles that maximize the sensing quality. \citet{fan2019joint} applies the reverse combinatorial auction to motivate taxi drivers to perform the sensing tasks following the scheduled trajectories. \citet{xu2019ilocus} and \citet{chen2020pas} navigate the idle taxis mounted with sensors to achieve desirable sensing quality. A comparable incentive mechanism is proposed to motivate drivers to follow the navigation. \citet{asprone2021vehicular} calculates a set of $\varepsilon$-minimum routes for partial vehicles such that sensing coverage is maximized, in which the cost of any $\varepsilon$-minimum route is less than that of the minimum cost route multiplied by a predefined parameter $\varepsilon$. 
	
	The existing studies assume that a taxi can only be dedicated to either conducting sensing tasks or serving passengers. However, this assumption does not accurately conform to reality and tends to overestimate the success rate of drivers in undertaking sensing tasks. It overlooks the opportunity costs involved in conducting regular work, which could affect drivers' willingness to undertake sensing tasks. This observation motivates us to develop an operational strategy that integrates sensing tasks, though small in quantity compared to ride-hailing orders, into the regular workflow of drivers. In this approach, we assume the drivers are rational and choose sensing tasks based on their net utility. To our knowledge, this approach is novel in the literature on taxi-based mobile sensing.
	
	\subsection{Coordination of multiple tasks in the taxi network}
	The literature on the coordination of multiple heterogeneous crowd-sourcing tasks within the taxi/ride-hailing network is new and fast-growing. An early attempt is the Share-a-ride problem (SARP) in which a mixed integer linear program is formulated to plan the set of sharing routes for passengers and parcels using the same fleet of taxis \citep{li2014share}. It is noted that people and freight delivery differ in many aspects such as price and service level. Nevertheless, the passengers should be prioritized in the person-freight shared taxi network. \citeauthor{chen2017using} propose a heuristic strategy for collecting e-commerce reverse flow using taxis. For the sake of reducing environmental impact and promoting the revenue of drivers, only taxis with passengers on board are allowed to collect returned goods. The freight delivery with the mixed fleet of regular and occasional vehicles is studied as the Vehicle Routing Problem with Occasional Drivers (VRPOD) by \cite{archetti2016vehicle}. In such a problem, the number of regular drivers is unlimited but they may request high cost. Vice versa for the occasional drivers. Extensions to online and bundle delivery problems are made by \cite{archetti2021online} and \cite{mancini2022bundle}, respectively. There is also a vast of studies using the auction-based approach to commission delivery tasks to shippers (\cite{zou2023designing} for example), their technical details will be reviewed in the next section. Despite its limitation in scalability \citep{qi2018shared}, the potential economic benefit and operational flexibilities of the crowdsourced delivery make it particularly attractive to shippers. For the state-of-the-art joint person and freight delivery with taxis, one may refer to \cite{alnaggar2021crowdsourced}.
	
	In addition to the integrated operation strategies for multiple crowd-sourcing tasks in the taxi network, one may also be curious about the business modes of their alliance. Unfortunately, no study is available on how different types of ride-hailing tasks are coordinated by the same operator. However, a similar business mode in which different ride-hailing platforms for the same tasks are coordinated within one integrator has been studied by a handful of researchers. \citet{zhou2022competition} first studies the third-party platform integration in ride-sourcing markets and assumes the integrator could directly control the vehicle-ride request matching process, despite which ride-hailing platform a vehicle is affiliated to. Thus, the integrator is responsible for maximizing the number of realized hailing orders and social welfare respecting the equilibrium among ride-hailing platforms. \citet{li2022allocation} observe that the ride-hailing platforms may create a so-called `artificial scarcity' market phenomenon by sacrificing the order completion rate for high profit, which is caused by the inappropriate pricing strategy of the integrator. They propose a Stackelberg game model for pricing to remove artificial scarcity. \citet{bao2023mathematical} address a ride-hailing order assignment approach for the third-party integrator aiming at minimizing the waiting time of passengers. 
	
	Despite the abundance of literature on the market environment, business operation, and stakeholders' decision-making problems of the ride-hailing platform \citep{wang2019ridesourcing}, there is limited study on the coordination of taxis for heterogeneous tasks by one operator. 
	Before stepping into the technical details, the ride-hailing platforms are curious about problems such as whether will they benefit from the crowd-sourced tasks, how the crowd-sourced tasks are released dynamically, how the crowd-sourced tasks and ride-hailing orders are organized to guarantee the service level for passengers, how the targeted drivers are selected, how to commission the tasks to the appropriate drivers to avoid reluctance, etc. Our work will answer this operational problem when coordinating the sensing tasks and ride-hailing orders in the same taxi network. 
	\subsection{Auction-based task allocation in mobility}
	Auction-based mechanism designs have been extensively used in various mobility task allocation problems such as shared rides matching \citep{YAN2021102351, BIAN202077}, parcel delivery \citep{Yafei2022auction} and mobile sensing task commission \citep{fan2019joint}. The main motivation for using an auction-based assignment approach is to ensure the welfare of all stakeholders involved in the trading process. The auction approaches could be broadly classified into: (i) combinatorial auction, in which a bidder is allowed to submit price for multiple tasks in each round of auction \citep{HAMMAMI2021204}; (ii) sequential auction, in which the targets are released dynamically over time and are auctioned sequentially \citep{MOCHON2022PPP, XIANG2023advanced}; and (iii) double auction, in which both the supply and demand sides are allowed to make offers and the auction is settled by the auctioneer \citep{Xu2017truthful, Li2020pricing}. There are also research works that fall into the intersection of two groups, for example, \citet{karamanis2020assignment} in using the combinatorial double auction for solving the ride-sharing assignment and pricing problems. 
	%
	
	While numerous studies address task allocation issues in mobility services, the auctioning of mobile sensing tasks to ride-hailing drivers remains largely unexplored. The distinctive characteristics of these tasks, such as their sparse spatial distribution, flexible time requirements, and limited data volume, necessitate tailored approaches when implementing auctions. Moreover, concerns regarding budget deficits from both data users and ride-hailing platform operators in adopting taxi-based mobile sensing should also be addressed. Otherwise, data users may opt to utilize their dedicated vehicles for sensing to mitigate costs. To tackle these challenges, we propose a novel auction-based mechanism for assigning mobile sensing tasks that excludes non-economical matching combinations and refines the payment rule.
	
	\section{Coordination of ride-hailing and mobile sensing tasks}\label{section:model}
	
	This section delves into the coordination of ride-hailing orders and mobile sensing tasks within a single ride-hailing taxi fleet. The platform operates with two types of taxis: Type-A vehicles exclusively handle trip requests, while Type-B vehicles are equipped with sensing devices and can accommodate both trip requests and mobile sensing tasks. Denoting the sets of Type-A and Type-B vehicles as $\mathcal{\overline{D}}$ and $\mathcal{\widetilde{D}}$ respectively, it's worth noting that the fleet size of Type-B vehicles is typically smaller due to the higher costs associated with purchasing and installing sensing devices. The union of these sets, $\mathcal{D} = \mathcal{\overline{D}}\cup \mathcal{\widetilde{D}}$ represents the entirety of taxis registered with the ride-hailing platform.
	
	Consider a study area comprising various locations earmarked for monitoring, each corresponding to a pending sensing task. Given that the completion rate of sensing tasks hinges on the availability of active Type-B drivers, employing an auction-based task assignment method becomes favorable to motivate Type-B drivers to undertake sensing tasks. This incentivization is driven by the potential for Type-B drivers to earn a higher expected reward from sensing tasks compared to trip requests. The overall payoff across all tasks is governed by the budget $\Omega$. To distinguish between passenger delivery and sensing tasks, we segment the service time of Type-B vehicles, denoted as $\mathcal{T}$, into multiple plan cycles. Each cycle $T$ comprises three phases:
	(i) the ride-hailing task matching phase $T_1$, where ride-hailing tasks are assigned to both types of vehicles;
	(ii) the mobile sensing task bidding phase $T_2$, during which a Type-B vehicle failing to secure a ride-hailing task can bid on multiple sensing tasks; and
	(iii) the mobile sensing task assignment phase $T_3$, wherein tasks are allocated to Type-B vehicles through the auction-based approach.
	In this system, both Type-A and Type-B vehicles enjoy equal priority in $T_1$, and the operation of Type-A vehicles remains unaffected by the sensing task assignment in $T_2$ and $T_3$. The following notations are consistently utilized throughout this article.
	
	\begin{table}[htbp]
		\centering
		\scriptsize
		\caption{Key parameters and variables}
		\begin{tabularx}{\textwidth}{p{0.08\textwidth}X}
			\toprule[2pt]
			\multicolumn{2}{l}{Indices and sets} \\
			\midrule
			$\mathcal{D}_t$ & Set of available drivers in the $t$-th interval (including both the Type-A and Type-B drivers) \\
			$\mathcal{\widetilde{D}}_t$ & Set of available Type-B drivers in the $t$-th interval ($\mathcal{\widetilde{D}}_t \subseteq \mathcal{D}_t$)\\
			$\mathcal{R}_t$ & Set of available riders in the $t$-th interval \\
			$\mathcal{K}_t$ & Set of task requests released in the $t$-th interval \\
			$A$ & Set of region in geographic network \\
			\midrule
			\multicolumn{2}{l}{Parameters and constants} \\  
			\midrule
			$p_{r}^{s}$ & Base price charged from a matched rider \\
			$L_0^{s}$ & Travel distance covered by the initial fare \\
			$t_0^{s}$ & Travel time covered by the initial fare \\
			$l_0$ & Average service distance of all trip requests \\
			$\beta_1$ & Incremental cost per unit travel distance when the distance exceeds $L_0^{s}$ \\
			$\beta_2$ & Incremental cost per unit time when the time exceeds $t_0^{s}$ \\
			$\overline{V}$ & Average vehicle speed \\
			$L_{ub}$ & Maximum pick-up distance \\
			$\alpha$ & Driver’s expected payoff per unit travel distance \\
			$\overline{b}$ & The upper bound of bid value \\
			$\underline{b}$ & The lower bound of bid value \\
			$\mu$ & Coefficient for calculating opportunity cost in the task assignment \\
			$c_q$ & Operating cost per unit travel distance when using the dedicated vehicle \\
			$C$ & Base reward for a task request \\
			$\Omega$ & Overall budget for sensing tasks \\
			\midrule
			\multicolumn{2}{l}{Decision variables} \\
			\midrule
			$x_{dr}^{t}$ &  A binary variable equals to 1 if driver $d \in \mathcal{D}_t$ is matched with a rider $r \in \mathcal{R}_t$ are matched in the $t$-th interval, and 0 otherwise \\
			$x_{dk}^{t}$ &  A binary variable equals to 1 if driver $d \in \mathcal{\widetilde{D}}_t$ is matched with a sensing task $k \in \mathcal{K}_t$ in $t$-th interval, and 0 otherwise \\
			\bottomrule[2pt]
		\end{tabularx}
		\label{tab:my_label}
	\end{table}

	During phase $T_2$, the ride-hailing platform releases a specific number of sensing tasks, determined by the available number of drivers and the remaining budget. Each task forms an auction venue. Eligible drivers, including those Type-B drivers who shift to the unoccupied state following the completion of a sensing task or trip order and those who fail to match any trip order, are allowed to bid on multiple tasks. As the bidding concludes, the ride-hailing platform gathers bids from each auction venue. Subsequently, the platform is responsible for determining the winners and payoffs for each sensing task through a combinatorial auction process.

	\paragraph{Opportunity cost of drivers} In addition to the direct rewards from sensing tasks and trip requests, drivers often factor in opportunity costs when deciding whether to undertake a trip request or MST. These opportunity costs can be roughly estimated by considering a driver's expected waiting time at the destination for a new trip request or MST. Typically, drivers are reluctant to accept remote requests, especially when pickups and drop-offs are in suburban or rural areas, due to the extended time required to search for subsequent MSTs or trip requests in regions with lower travel demand. The expected waiting time for the next trip request depends on both the supply and demand at the destination, encompassing the number of available vehicles and trip requests. These values can be estimated from historical data or calculated on the fly.
	
	Let $n_{aT}$ denote the estimated number of trip requests in the region $a$ during plan cycle $T$, $n_{aT}^{taxi}$ denote the estimated count of available drivers in the same region and plan cycle, and $T_0$ indicate the duration of each cycle. Thus, $\frac{n_{aT}}{n_{aT}^{taxi}+1}$ signifies the anticipated number of trip requests allocated to each driver if they head towards region $a$ during plan cycle $T$. The expression $\frac{n_{aT}}{T_0 (n_{aT}^{taxi}+1)}$ represents the average number of trip requests matched per driver per unit of time. Consequently, its inverse $\frac{T_0 (n_{aT}^{taxi}+1)}{n_{aT}}$ corresponds to the duration between completing one trip order and being matched with the next trip order in the region $a$, effectively denoting the estimated cruising time in that region.  However, if $n_{aT}$ is exceedingly small, $\frac{T_0 (n_{aT}^{taxi}+1)}{n_{aT}}$ can become extremely large. To avoid such a situation, we may impose an upper bound $f_m$ on the opportunity cost.
	
	Therefore, the opportunity cost for trip requests is calculated by
	\begin{equation}\label{opptm}
		f(r_d) = \min \left\{f_m, \quad \alpha \overline{V} \cdot \frac{T_0 (n_{a_{r_d}T}^{taxi}+1)}{2 n_{a_{r_d}T}}\right\}, 
	\end{equation}
	where $\alpha$ is driver's payoff per unit travel distance, $\overline{V}$ is the average vehicle speed in urban areas, and $a_{r_d}$ is the area where the destination $r_d$ is located.
	
	The procedure for calculating the opportunity cost of accepting a sensing task is as follows: When a Type-B driver accepts an MST, she opts out of the opportunity to pick up any potential riders near her current location. Given that MSTs typically occur in remote areas, we can simplify the determination of opportunity cost by assuming it relies solely on the distance between the driver's current location and the point of interest (POI) associated with the MST. Therefore, the opportunity cost $g(l)$ can be expressed as a monotonically increasing function of this distance:
	\begin{equation}\label{oppta}
		g(l) = \max \left\{0, \quad\mu(l - l_0)\right\}\quad \mu>0, 
	\end{equation}
	where $l$ and $l_0$ denote the actual service distance of the trip request and the average service distance of all trip requests, respectively.
	\subsection{Trip matching problem}
	
	Consider a driver-rider pair $(d, r)$. When matched by the ride-hailing platform, the fare charged to the rider comprises several components: the initial hailing fare, additional distance- and time-based fares. This cost can be expressed as:
	\begin{equation}
		p_r = p_r^{s}+\max\left\{0, \beta_1(L_r - L_0^{s})\right\}+\max\left\{0, \beta_2(t_r - t_0^{s})\right\},
	\end{equation} where $p_r^{s}$ represents the initial hailing fare set by the ride-hailing platform. The second term accounts for any additional distance-based fare when the actual distance traveled $L_r$ exceeds the base distance $L_0^{s}$.
	The third term represents any additional time-based fare when the actual travel time $t_r$ exceeds the base time $t_0^{s}$. $L_0^{s}$ and $t_0^{s}$ denote the base distance and time covered by the base taxi fare, respectively.
	
	On the driver's side, the earning to pick up the rider $r$ encompasses both the distance-based fare and the additional subsidy equal to the opportunity cost. Such subsidies encourage drivers to take remote trip requests, thereby increasing the number of successful matches and regulating the distribution of vehicles across different regions \citep{you2023order}.
	
	\begin{equation}
		p_{dr} = \alpha L_r + f(r_d).
	\end{equation} Here, $\alpha L_r$ represents the distance-based cost incurred by the driver for traveling the distance $L_r$ to pick up the rider. $f(r_d)$ denotes the subsidy associated with opportunity cost incurred by the driver's decision to accept the ride request from the specific rider $r_d$.
	%
	%
	%
	
	Following the idea of \cite{agatz2011dynamic}, our objective function is minimizing the total pickup distance while maximizing the number of matched driver-rider pairs. These objectives can also be regarded as maximizing the total pickup distance saving. Let $L_{dr}$ represent the pickup distance in time slot $t \in T$. The saved distance for a driver-rider pair is calculated as $\sigma_{dr}^{t} = \max_{r\in \mathcal{R}_t} \left\{L_{dr}\right\} - L_{dr}$. We formulate the trip-matching problem as follows:
	
	[TM-P1]
	\begin{align}
		\max &\sum_{d\in\mathcal{D}_t}\sum_{r\in\mathcal{R}_t}\sigma_{dr}^{t}x_{dr}^{t} \label{tmobj} \\
		\hbox{s.t.}~~ &\sum_{d\in\mathcal{D}_t}x_{dr}^{t}\leq 1, \quad\forall r\in\mathcal{R}_t \label{tmcon1} \\
		&\sum_{r\in\mathcal{R}_t}x_{dr}^{t}\leq 1, \quad\forall d\in\mathcal{D}_t \label{tmcon2} \\
		&(p_r^{t} - p_{dr}^{t})x_{dr}^{t} \geq 0, \quad \forall d\in\mathcal{D}_t, r\in\mathcal{R}_t \label{tmcon3} \\
		&\sigma_{dr}^{t} = \max_{r\in \mathcal{R}_t} \left\{L_{dr}\right\} - L_{dr}, \quad \forall d\in\mathcal{D}_t, r\in\mathcal{R}_t \label{tmcon4} \\
		&L_{dr}x_{dr}^{t} \leq L_{ub}, \quad \forall d\in\mathcal{D}_t, r\in\mathcal{R}_t \label{tmcon5} \\
		&x_{dr}^{t} \in \left\{0, 1\right\}, \quad \forall d\in\mathcal{D}_t, r\in\mathcal{R}_t \label{tmcon6}
	\end{align}
	Constraints \eqref{tmcon1} and \eqref{tmcon2} ensure that each driver and rider can be assigned only once. Constraint \eqref{tmcon3} excludes assignments resulting in negative revenues, while Constraint \eqref{tmcon5} prevents assigning riders to drivers who are excessively faraway from them.
	
	The trip-matching problem can be viewed as a bipartite graph-matching problem, noting Constraints \eqref{tmcon3}-\eqref{tmcon5} can be calculated beforehand. This problem can be efficiently solved using the Kuhn-Munkres (KM) algorithm, which finds the maximum weight perfect matching on the bipartite graph $G(\mathcal{N}_t, E_t)$ representing drivers, riders, and their potential matches. In this graph, the vertices $\mathcal{N}_t$ comprise the set of drivers $\mathcal{D}_t$ and riders $\mathcal{R}_t$, while the edges $E_t$ represent potential matches between drivers and riders, with weights corresponding to the cost $\sigma{dr}^{t}$ calculated by Equation \eqref{tmcon4}.
	
	To calculate the optimal matching, we initially eliminate the infeasible edges. Each vertex is then assigned an initial label: for one group of vertices, the labels are set to their maximum potential weights among all possible matches, while for the other group, the labels are set to 0. These labels are iteratively updated to identify augmenting paths in the bipartite graph. The algorithm terminates when no new matches are available. The KM algorithm is outlined in detail in Algorithm \ref{algorithm1}.

	\begin{algorithm}[h]
		\caption{Kuhn-Munkres Algorithm}\label{algorithm1}
		\KwIn{The set of drivers $\mathcal{D}_t$, the set of riders $\mathcal{R}_t$, weight matrix $W$}
		\KwOut{Optimal perfect matching}
		Construct a square matrix $W$ by adding virtual drivers and riders\;
		Initialize a match $M$ and labels for the drivers $l_x$ and the riders $l_y$ \;
		\For{$u \in \mathcal{D}_t$}{
			\While{{\bf True}}{
				Reset all vertices to unvisited state\;
				Find the augment path $p(u)$ for $u$\;
				\If{\textnormal{no augment path exists}}{
					{\bf break}\;
				}{
					\For{$v\in \mathcal{R}_t$}{
						\If{\textnormal{$v$ not visited} \enspace{\bf and}\enspace $l_x[u]+l_y[v]==W[u,v]$}{
							Find a augment path $p(v)$ for $v$\;
							\If{\textnormal{$v$ is not matched} {\bf or} \textnormal{$p(v)$ exists}}{
								Augment path $p(u)\leftarrow (u,v)$\;
								$M \leftarrow (u,v)$\;
							}
						}
					}
				}
				Find the tuple $(u,v)$ that $u$ is visited but $v$ is not\;
				Update each visited vertex's label with minimum $\Delta = l_x[u]+l_y[v]-W[u,v]$\;
			}
		}
	\end{algorithm}
	
	The variable $\sigma_{dr}^{t}$ may be dropped by reformulating the trip matching problem. For any rider $r$ who fails to match any driver, we may impose a penalty term on the platform's revenue. This term is parameterized by $M_k= \max_{r\in \mathcal{R}_t} \left\{L_{dr}\right\}$, which quantifies the system's (maximum) loss on not serving rider $r$. Then we may reach the following trip matching problem.
	
	[TM-P2]
	\begin{align}
		\min &=\sum_{d\in\mathcal{D}_t}\sum_{r\in\mathcal{R}_t} L_{dr}x_{dr}^{t} + \sum_{r\in\mathcal{R}_t}M_k(1 - \sum_{d\in\mathcal{\widetilde{D}}_t}x_{dr}^{t}) \label{ta3obj} \\
		\hbox{s.t.}~~&\text{Eqs.}\eqref{tmcon1}-\eqref{tmcon3}, \eqref{tmcon5},  \eqref{tmcon6} 
	\end{align}
	We note [TM-P1] and [TM-P2] result in exactly the same solutions.
	
	\subsection{The primal VCG-based MST assignment (VCG-MST) mechanism}
	The primary motivation for data users to utilize Type-B vehicles for MSTs lies in reducing their costs associated with purchasing and maintaining dedicated sensing vehicles. Hence, the objective of the MST assignment problem is to maximize savings on expenses related to sensing tasks.
	
	Conversely, a driver $d$ interested in bidding for a MST $k$ may anticipate a higher payoff compared to a hailing order covering the same distance as the MST. In the $t$-th interval, we define the \textit{hidden valuation} as $d$'s expected payoff for a hailing order with a distance of $l_{dk}$, calculated by:
	\begin{equation}
		\label{vdeq}
		\underline{v}_{dk}^t=\alpha l_{dk}+g(l_{dk})=
		\begin{cases}
			\alpha l_{dk} & l_{dk}\leq l_0 \\
			\alpha l_0 + (\alpha+\mu)(l_{dk}-l_0) & l_{dk}> l_0
		\end{cases}
	\end{equation} 
	where $l_0$ denotes the average service distance of all trip requests. 
	
	\begin{remark}
		In this study, we regard the hidden valuation as a driver's estimation of payoff to a MST. If $l_{dk}\leq l_0$, the task is not far from the driver, and the driver's expected payoff from the MST can be calculated similarly as a trip request. However, when $l_{dk} > l_0$, the driver may need to drive away from the urban area to perform the sensing tasks, requiring additional compensation for the inability to serve more passengers. This observation motivate us to add this compensation $\mu(l_{dk}-l_0)$ to the driver's valuation.
	\end{remark}
	
	\paragraph{Winner selection} 
	As noted at the beginning of this section, the MSTs are allocated to Type-B drivers through a series of auction venues, each corresponding to a planning cycle. The platform releases a set of MSTs $\mathcal{K}_T$ at the beginning of the task assignment phase $T_2$. A driver $d\in \mathcal{\widetilde{D}}_t$ has the option to submit bids for a limited number of MSTs at the same unit price. In this study, any valid bid $\mathscr{B}_{dk}^t$ must satisfy the following inequality:
	\begin{equation}\label{bid_lbub}
		\alpha \leq \underline{b} \leq \mathscr{B}_{dk}^t \leq \overline{b}
	\end{equation}
	where $\underline{b}$ and $\overline{b}$ denote the lower and upper bounds of the unit price per distance, respectively. Let $\alpha$ represent the expected payoff per unit distance for serving a passenger, and $c_q$ denote the unit cost of a dedicated sensing vehicle for completing a task. It is evident that $\underline{b}\geq \alpha$ and $\overline{b}\leq c_q$. If driver $d$ opts not to bid for an MST $k$, we set $\mathscr{B}_{dk}^{t}=+\infty$. It's important to note that all drivers independently submit their bids. Therefore, driver $d$'s stated valuation for MST $k$ is calculated as follows:
	\begin{equation}\label{v_state}
		\tilde{v}^t_{dk} = C+\mathscr{B}^t_{dk} l_{dk}
	\end{equation} where $C$ is the driver's base payoff on completing an MST.
	
	We may calculate the intersection of the state valuation and the hidden valuation. Clearly, $l^*=\frac{C+\mu l_0}{\alpha+\mu-\mathscr{B}_{dk}^t}$. When driver $d$'s bid $\mathscr{B}_{dk}^t<\alpha+\mu$ and the distance $l_{dk} < l^*$, the stated valuation must be greater than the hidden valuation, which means that driver $d$ can earn at least as much as expected and doesn't need any compensation. Figure \ref{fig:comp} illustrates two relationship of the hidden valuation (\eqref{vdeq}) and the stated valuation \eqref{v_state} under different parameter settings.
	
	\begin{figure}[htbp]
		\centering
		\includegraphics[width=\textwidth]{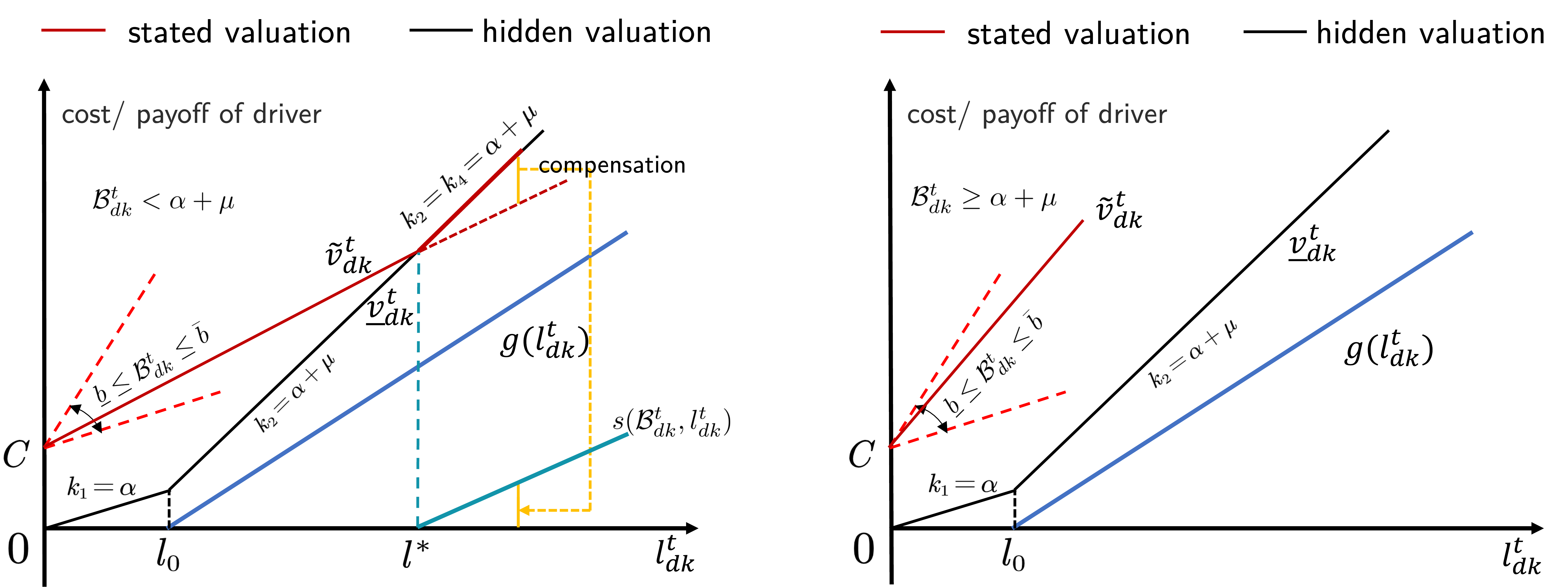}
		\caption{A driver's stated valuation and hidden valuation to a task. Left: When $\mathscr{B}_{dk}^t<\alpha+\mu$ and $l_{dk} \geq l^*=\frac{C+\mu l_0}{\alpha+\mu-\mathscr{B}_{dk}^t}$, $\tilde{v}_{dk}^t < \underline{v}_{dk}^t$. The driver's valuation should be adjusted by $s(\mathscr{B}_{dk}^t, l_{dk})$. Right: $\mathscr{B}_{dk}^t\geq \alpha+\mu$, $\tilde{v}_{dk}^t \geq \underline{v}_{dk}^t$. The stated valuation is always above the hidden valuation.}
		\label{fig:comp}
	\end{figure}
	However, the state valuation is less than the driver $d$'s hidden valuation for the MST when $\mathscr{B}_{dk}^t<\alpha+\mu$ and $l_{dk} \geq l^*$. For a rational driver, the final earning for any MST should always be no less than the hidden valuation $\underline{v}^t_{dk}$. By simple algebraic analysis, we have:
	\begin{equation}
		\begin{cases}
			\tilde{v}_{dk}^t < \underline{v}_{dk}^t, & \text{ if } \mathscr{B}_{dk}^t<\alpha+\mu \text{ and } l_{dk}> \frac{C+\mu l_0}{\alpha+\mu-\mathscr{B}_{dk}^t}\\ 
			\tilde{v}_{dk}^t \geq  \underline{v}_{dk}^t,  & \text{ otherwise. }
		\end{cases}
	\end{equation}
	Therefore, we introduce a compensating function $s(\mathscr{B}_{dk}^t, l_{dk}^t)$ to adjust the driver's valuation. $s(\mathscr{B}_{dk}^t, l_{dk})$ is given by $\max\{\underline{v}_{dk}^t -\tilde{v}_{dk}^t , 0\}$. We have 
	\begin{equation}\label{compfunc}
		s(\mathscr{B}_{dk}^t, l_{dk}) =
		\begin{cases}
			\alpha l_{dk}+\mu(l_{dk}-l_0) - \mathscr{B}_{dk}^t l_{dk}, &\text{ if } \mathscr{B}_{dk}^t<\alpha+\mu \text{ and } l_{dk}> \frac{C+\mu l_0}{\alpha+\mu-\mathscr{B}_{dk}^t}\\ 
			0, & \text{otherwise.}
		\end{cases}
	\end{equation}
	
	With this additional term, driver $d$'s adjusted valuation for task $k$ now becomes:
	\begin{equation}\label{tvdeq}
		v_{dk}^{t}= \tilde{v}_{dk}^t + s(\mathscr{B}_{dk}^t, l_{dk}) =C + \mathscr{B}_{d}^{t}l_{dk} + s(\mathscr{B}_{d}^{t}, l_{dk})
	\end{equation}
	where $\mathscr{B}_{d}^{t}$ is the bid of driver $d$ for any task request in $t$'s decision cycle.
	
	In the winner selection phase $T_3\in T$, a Type-B driver may be assigned either a hailing order or an MST. The following inequality ensured that the driver could handle at most one task or order simultaneously:
	\begin{equation}\label{ta1con3}
		x_{dr}^{t} \leq 1-x_{dk}^{t}, \quad\forall d,r,k \in\mathcal{\widetilde{D}}_t,\mathcal{R}_t,\mathcal{K}_T; \quad t\in T.
	\end{equation}
	Here, $\mathcal{\widetilde{D}}_t$ represents the set of available Type-B vehicles in time slot $t$, $\mathcal{R}_t$ denotes the set of available riders in time slot $t$, and $\mathcal{K}_T$ indicates the set of MSTs released in time slot $T$.
	
	In this study, the \textit{sensing cost saving} $\delta_{dk}^{t}$ quantifies the cost reduction achieved by assigning the MST to a Type-B driver instead of dispatching a dedicated vehicle. This cost reduction is computed as follows:
	\begin{equation}\label{ta1con4}
		\delta_{dk}^{t} = c_q \cdot l_{qk} - v_{dk}^{t}.
	\end{equation}
	Here, $c_q$ represents the unit cost per kilometer when using a dedicated vehicle, while $l_{qk}$ denotes the distance from the nearest depot of dedicated sensing vehicles to the MST $k$, which can be approximated by the shortest distance between them.
	
	An auction venue will result in a winner if at least one driver submits a valid bid. In mathematical terms, this condition can be expressed as:
	\begin{equation}\label{ta1con4_1}
		\tau_{k}^{t}=1, \quad \text{if} \quad \exists \delta_{dk}^{t} > -\infty, \enspace \forall k\in \mathcal{K}_T,
	\end{equation}
	where $\tau_{k}^{t}$ is an indicator variable that equals 1 if the MST $k$ is assigned to any driver, and 0 otherwise.
	
	We present the following integer linear programming model to select the winners of the MSTs for each round $t$:
	
	[MST-P1]
	\begin{align}
		\max &\quad\Pi=\sum_{k\in \mathcal{K}_T}\sum_{d\in \mathcal{\widetilde{D}}_t} \delta_{dk}^{t}x_{dk}^{t} \label{ta1obj} \\
		\hbox{s.t.}~~&\text{Eqs.}\eqref{tvdeq}-\eqref{ta1con4_1} \nonumber \\
		&\sum_{k\in\mathcal{K}_T}x_{dk}^{t}\leq 1, \quad\forall d\in \mathcal{\widetilde{D}}_t \label{ta1con1} \\
		& \sum_{d\in\mathcal{\widetilde{D}}_t}x_{dk}^{t}\leq 1, \quad\forall k\in \mathcal{K}_T \label{ta1con2} \\
		&\sum_{d\in\mathcal{\widetilde{D}}_t}x_{dk}^{t}\geq \tau_{k}^{t}, \quad\forall k\in \mathcal{K}_T \label{ta1con2_1} \\
		& x_{dk}^{t}, x_{dr}^{t},\tau_{k}^{t} \in \left\{0,1\right\} \label{ta1con10}
	\end{align} 
	In this formulation, Constraints \eqref{ta1con1} and \eqref{ta1con2} ensure that each driver or MST can be matched only once. Constraint \eqref{ta1con2_1} guarantees that any task request will be assigned when there are valid bidders for it.
	
	As previously discussed, Type-B vehicles are considered to be `redundant' vehicles if they fail to match any rider in $T_1$. These vehicles are eligible to participate in both the trip matching and MST assignment problems in $T_3$. To motivate Type-B drivers, the ride-hailing platform may prioritize MSTs in ${T_3}$. In other words, the MST assignment problem is solved first. Any unoccupied Type-A and Type-B drivers are subsequently matched to ride orders by addressing the trip matching problem, while those who win the auction are filtered out.
	
	\paragraph{Payment determination} After assigning task requests to drivers, the next problem is to determine a reasonable payoff for each driver, which is calculated by the ride-hailing platform. A straightforward payment rule derives from the one-sided VCG mechanism. In the VCG-MST mechanism, the payoff of each driver $d$ is determined by:
	\begin{equation}\label{vcgrev}
		p_{d}^{t} = v_{dk}^{t} + (\Pi(\mathcal{\widetilde{D}}_t) - \Pi(\mathcal{\widetilde{D}}_t \backslash \left\{d\right\})) \quad \forall d \in \mathcal{\widetilde{D}}_t
	\end{equation}
	where $\Pi(\mathcal{\widetilde{D}}_t)$ and $\Pi(\mathcal{\widetilde{D}}_t \backslash \left\{d\right\})$ represent the optimal objective function value with and without driver $d$. A driver's utility $u_d^{t}$ is calculated as the difference between the payoff and true valuation:
	\begin{equation}\label{vcguti}
		u_d^{t} = p_{d}^{t} - v_{dk}^{t}
	\end{equation}
	The aforementioned winner selection and payment determination process are summarized in Algorithm \ref{algorithm2}.
	
	\begin{algorithm}[h]
		\caption{VCG-MST assignment mechanism}\label{algorithm2}
		\KwIn{The driver set $\mathcal{D}_t$, the task request set $\mathcal{K}_{T}$, the rider set $\mathcal{R}_t$, sensing cost saving matrix $W$}
		\KwOut{Trip matching result $X_{\mathscr{D}\mathscr{R}}$, MST assignment result $X_{\mathscr{W}\mathscr{K}}$, vector of payments to MSTs $\textbf{p}$}
		Solve the MST assignment model to obtain the optimal solution $X_{\mathscr{W}\mathscr{K}}$, the set of winners $X_{\mathscr{W}}$ and the set of selected task requests $X_{\mathscr{K}}$\;
		Calculate the optimal objective function value $\Pi(\mathcal{\widetilde{D}}_t)$\;
		Solve the trip matching model to obtain the optimal solution $X_{\mathscr{D}\mathscr{R}}$, the set of matched drivers $X_{\mathscr{D}}$ and the set of matched riders $X_{\mathscr{R}}$\;
		\For{$(d,r)\in X_{\mathscr{D}\mathscr{R}}$}
		{\If{$d\in X_{\mathscr{W}}$}
			{Search for another idle driver $d_2$ who satisfies $d_2 = \argmin_{j\in \mathcal{D}_t\backslash(X_{\mathscr{W}}\cup X_{\mathscr{D}})} \left\{L_{jr}\enspace\big | \enspace L_{jr}\leq L_{ub}\right\}$\;
				\eIf{$\exists d_2 \in \mathcal{D}_t\backslash(X_{\mathscr{W}}\cup X_{\mathscr{D}})$}
				{replacing the incumbent driver $d$ for rider $r$ by $d_2$\;}
				{$X_{\mathscr{D}\mathscr{R}} \leftarrow X_{\mathscr{D}\mathscr{R}}\backslash \left\{(d,r)\right\}$\;
					$X_{\mathscr{D}} \leftarrow X_{\mathscr{D}}\backslash \left\{d\right\}$\;
					$X_{\mathscr{R}} \leftarrow X_{\mathscr{R}}\backslash \left\{r\right\}$\;}
			}
		}
		\For{$d\in X_{\mathscr{W}}$}
		{Remove the row for driver $d$ from the matrix $W$\;
			Re-calibrate the MST assignment model to obtain the optimal objective function value $\Pi(\mathcal{\widetilde{D}}_t \backslash \left\{d\right\})$\;
			Calculate the driver $d$'s payoff $p_d^{t}=v_{dk}^{t}+(\Pi(\mathcal{\widetilde{D}}_t) - \Pi(\mathcal{\widetilde{D}}_t \backslash \left\{d\right\}))$\;
		}
		\textbf{Return} \quad $X_{\mathscr{D}\mathscr{R}}, X_{\mathscr{W}\mathscr{K}}, \textbf{p}$
	\end{algorithm}
	
	The VCG-MST mechanism is proved to satisfy favorable economic properties such as Individual Rationality, Allocative Efficiency (AE), and Incentive Compatibility.
	\begin{proposition}{\bf Individual Rationality}
		Any driver who bids an MST would not get negative utility.
	\end{proposition}
	\begin{proof}
		From the Eq.\eqref{vcgrev} and Eq.\eqref{vcguti}, the driver's utility is $\Pi(\mathcal{\widetilde{D}}_t) - \Pi(\mathcal{\widetilde{D}}_t \backslash \left\{d\right\})$. If a driver $d$ participates in sensing activity, the optimal value of Eq.\eqref{ta1obj} is at least not less than that in Eq.\eqref{ta1obj}, then we have $\Pi(\mathcal{\widetilde{D}}_t) \geq \Pi(\mathcal{\widetilde{D}}_t \backslash \left\{d\right\})$, which leads to $u_d^{t}\geq 0$. This completes the proof.
	\end{proof}
	
	\begin{proposition}{\bf Allocative Efficiency}
		The VCG-MST auction mechanism for MST assignment is allocative efficiency.
	\end{proposition}
	\begin{proof}
		The objective function Eq.\eqref{ta1obj} aims to maximize the saving on sensing cost. In this MST problem, the VCG-MST auction mechanism operates as a price-only reverse auction, as it solely collects price bids from participants. Therefore, if Eq.\eqref{ta1obj} yields the maximum value, it indicates that the assignment mechanism has achieved allocative efficiency.
	\end{proof}
	
	\begin{proposition}{\bf Incentive Compatibility}
		By the VCG-MST mechanism, submitting the true price to the platform is a dominant strategy for any driver despite how other drivers submit their bids.
	\end{proposition}
	\begin{proof}
		We assume that everyone in $\mathcal{\widetilde{D}}_t \backslash \{d\}$ submits price bidding truthfully. If driver $d$ truthfully submits her price bidding $\mathscr{B}_{d}^{t}$, then her payoff is $p_d^{t}=v_{dk}^{t} + (\Pi(\mathcal{\widetilde{D}}_t) - \Pi(\mathcal{\widetilde{D}}_t \backslash \left\{d\right\}))$. If the driver $d$ submits her bidding price $\hat{\mathscr{B}}_{d}^{t} \neq \mathscr{B}_{d}^{t}$ untruthfully,  $\hat{v}_{dk}^{t} \neq v_{dk}^{t}$, then his payoff is $\hat{p}_d^{t}=\hat{v}_{dk}^{t}+(\hat{\Pi}(\mathcal{\widetilde{D}}_t) - \Pi(\mathcal{\widetilde{D}}_t \backslash \left\{d\right\}))$, and 
		\begin{equation*}
			\hat{\Pi}(\mathcal{\widetilde{D}}_t) = \hat{\delta}_{dk}^{t}+\sum_{j\in \mathcal{K}_t\backslash \{k\}}\sum_{l\in \mathcal{\widetilde{D}}_t\backslash \{d\}}\delta_{lj}^{t}x_{lj}^{t}
		\end{equation*}
		is the new objective function value. In light of the \citet{cheng2023novel}, we assume that driver $d$ can be matched, and obtain a higher payoff by submitting untruthfully. Then, we have:
		\begin{equation*}
			\hat{p}_d^{t}-v_{dk}^{t} > p_d^{t}-v_{dk}^{t} = \Pi(\mathcal{\widetilde{D}}_t) - \Pi(\mathcal{\widetilde{D}}_t \backslash \left\{d\right\}).
		\end{equation*}
		
		Substituting for $\hat{p}_d^{t}$ can get
		\begin{align*}
			\quad &\hat{v}_{dk}^{t}+\hat{\Pi}(\mathcal{\widetilde{D}}_t) - \Pi(\mathcal{\widetilde{D}}_t \backslash \left\{d\right\})-v_{dk}^{t} > \Pi(\mathcal{\widetilde{D}}_t) - \Pi(\mathcal{\widetilde{D}}_t \backslash \left\{d\right\}) \\
			\Rightarrow\quad & \hat{v}_{dk}^{t} + \hat{\Pi}(\mathcal{\widetilde{D}}_t) -v_{dk}^{t} +c_q l_{qk}-c_q l_{qk}> \Pi(\mathcal{\widetilde{D}}_t) \\
			\Leftrightarrow\quad & \delta_{dk}^{t} - \hat{\delta}_{dk}^{t} + \hat{\Pi}(\mathcal{\widetilde{D}}_t) > \Pi(\mathcal{\widetilde{D}}_t) \\
			\Rightarrow\quad & \delta_{dk}^{t} - \hat{\delta}_{dk}^{t} + \hat{\delta}_{dk}^{t}+\sum_{j\in \mathcal{K}_t\backslash \{k\}}\sum_{l\in \mathcal{\widetilde{D}}_t\backslash \{d\}}\delta_{lj}^{t}x_{lj}^{t} > \Pi(\mathcal{\widetilde{D}}_t) \\
			\Rightarrow\quad & \delta_{dk}^{t} + \sum_{j\in \mathcal{K}_t\backslash \{k\}}\sum_{l\in \mathcal{\widetilde{D}}_t\backslash \{d\}}\delta_{lj}^{t}x_{lj}^{t} > \Pi(\mathcal{\widetilde{D}}_t).
		\end{align*}
		
		This contradicts that the optimal objective value is $\Pi(\mathcal{\widetilde{D}}_t)=\sum_{j\in \mathcal{K}_t\backslash \{k\}}$ $\sum_{l\in \mathcal{\widetilde{D}}_t\backslash \{d\}}\delta_{lj}^{t}x_{lj}^{t}+\delta_{dk}^{t}$, indicating that the driver $d$ cannot obtain higher payoff by submitting untruthfully. Therefore, the VCG-MST auction mechanism is incentive compatibility.
	\end{proof}

	\section{A refine budget control mechanism for MST assignment}\label{section:rbc}
	
	The VCG-based MST assignment mechanism offers a rational, efficient, and truthful approach to assigning MSTs within the ride-hailing platform. However, the resulting assignment plan may sometimes fail to meet budgetary requirements. Our numerical tests indicate instances where assigning MSTs to Type-B drivers proves to be more costly than utilizing dedicated vehicles for certain tasks. This could lead to budget deficits for the ride-hailing platforms, thereby affect their enthusiasms for incorporating mobile sensing services.
	
	To address this challenge and maintain control over the overall budget, we refine the winner determination and payment rules of the VCG mechanism. Additionally, we propose a budget control mechanism for MST assignment. The refined budget control MST assignment mechanism (RBC-MST) is designed to satisfy economic properties such as IR, IC, and BB.
	
	\subsection{Budget control for winner selection}
	The allocation of the mobile sensing budget directly impacts the sensing capabilities of the vehicles. A straightforward yet potentially risky approach involves allocating the entire budget at once. Under this method, the budget for each cycle fluctuates throughout the decision-making process as payments from previous bidding rounds are subtracted. However, this approach may lead to substantial opportunity costs \citep{tafreshian2022truthful}, particularly because complete information regarding the number and locations of MSTs may not be available initially.
	
	An alternative strategy is to incrementally invest the budget, distributing funds across each round of MST assignment and subsequently reallocating any remaining funds back to the pool. The amount of money invested in each round is primarily determined by the number of sensing tasks released during that round. The RBC-MST mechanism then determines the winners and payments based on the budgets allocated for each round of auctions.
	
	Let $\mathcal{K}_T^{r}$ denote the set of the remaining tasks In the $T$s round of MST assignment and $\mathcal{K}_{T}$ the subset of tasks released at the beginning of $T$. The budget for this round is then calculated by:
	\begin{equation}\label{budget_t}
		\Omega_T = \frac{\lvert\mathcal{K}_{T}\rvert}{\lvert\mathcal{K}_T^{r}\rvert}(\Omega - \sum_{i=0}^{T-1}\Theta_i)
	\end{equation}
	where $\Theta_i$ is the actual expense on tasks in round $i$ and $\Theta_0 = 0$.
	
	We are now ready to propose the following MST assignment problem by imposing an additional budget constraint on [MST-P1].
	
	[MST-P2]
	\begin{align}
		\min &\quad\Psi^{\prime}=\sum_{k\in \mathcal{K}_T}\sum_{d\in \mathcal{\widetilde{D}}_t} v_{dk}^{t}x_{dk}^{t} \label{ta2obj} \\
		\hbox{s.t.}~~&\text{Eqs.\eqref{tvdeq}, \eqref{ta1con3},} \eqref{ta1con1} - \eqref{ta1con10} \nonumber \\
		& \tau_{k}^{t} = 1, \quad \exists v_{dk}^{t} < +\infty \label{ta2con3}, \forall d\in \widetilde{D}_t,k\in\mathcal{K}_T \\
		&\sum_{k\in \mathcal{K}_T}\sum_{d\in \mathcal{\widetilde{D}}_t} (C+\overline{b}l_{dk}+s(\overline{b}, l_{dk}))x_{dk}^{t} \leq \Omega_T \label{ta2con4}, \forall d\in \widetilde{D}_t,k\in\mathcal{K}_T
	\end{align}
	where the Constraint \eqref{ta2con4} restricts the number of matching pairs in round $T$. 
	
	The drivers for MST tasks are selected through the following greedy procedure. Initially, Constraint \eqref{ta2con4} is relaxed when solving [MST-P2], and $\mathscr{W}\mathscr{K}_t$ represents the tuple set of ``winners'' along with the MSTs assigned to them. Though the payment to these drivers may potentially violate the budget constraint, they form the foundation for the final solution.
	
	Subsequently, we exclude the most ``expensive'' driver from the set, identified as the one with the largest true valuation. Mathematically,
	\begin{equation}\label{ta2con7}
		x_{d^{*}k^{*}}^{t} = 0, \quad (d^{*}, k^{*}) = \argmax_{(d,k)\in \mathscr{W}\mathscr{K}_t}\{v_{dk}^{t}\enspace\big| \enspace \sum\limits_{(d,k)\in \mathscr{W}\mathscr{K}_t}(C+\overline{b}l_{dk}+s(\overline{b}, l_{dk})) > \Omega_T \}
	\end{equation}
	The removed driver-MST pair is then added to a \textit{tabu list} $\mathscr{N}\mathscr{K}_t$, prohibiting its selection in future iterations. The [MST-P2] is then resolved with respect to the $\mathscr{N}\mathscr{K}_t$. This procedure is iterated until the budget constraint is satisfied. If all drivers eligible for task $k^{*}$ are in the tabu list, we may relax Constraint \eqref{ta2con3} based on the equation \eqref{ta2con3_relax}. 
	
	\begin{equation}\label{ta2con3_relax}
		\tau_{k^{*}}^{t} = 
		\begin{cases}
			0, & \textnormal{if} \quad \forall (d, k^{*})\in \mathscr{N}\mathscr{K}_t \enspace\textnormal{and}\enspace v_{dk^{*}}^{t} < +\infty \\
			1, & \textnormal{others}
		\end{cases}
	\end{equation}

	\subsection{A tailored payment rule}
	As with the VCG-MST mechanism, we denote the set of successfully matched driver-task pairs as $\mathscr{WK}_t$, with $\mathcal{W}_t$ representing the set of winners and $\mathcal{K}_t$ representing the set of task requests. We propose the following payment rule for the selected drivers. Driver $d$'s payoff is determined by:
	\begin{equation}\label{ta2payoff}
		p_d^{t} = 
		\begin{cases}
			C+\overline{b}l_{dk}+s(\overline{b}, l_{dk}), \quad &\Psi^{\prime}(\widetilde{D}_t\backslash\{d\}) < \Psi^{\prime}(\widetilde{D}_t) \\
			\min \bigg\{ v_{dk}^{t} + \big(\Psi^{\prime}(\widetilde{D}_t\backslash\{d\}) - \Psi^{\prime}(\widetilde{D}_t)\big), \enspace C+\overline{b}l_{dk}+s(\overline{b}, l_{dk}) \bigg\}, \quad &\Psi^{\prime}(\widetilde{D}_t\backslash\{d\}) \geq \Psi^{\prime}(\widetilde{D}_t)
		\end{cases}
	\end{equation}
	
	Assuming that $X(\widetilde{D}_t\backslash\{d\})$ and $X(\widetilde{D}_t)$ are solutions corresponding to the optimal objective value $\Psi^{\prime}(\widetilde{D}_t\backslash\{d\})$ and $\Psi^{\prime}(\widetilde{D}_t)$, respectively. We also utilize the marginal benefit of the winner $d$ to calculate his/her payoff, but this payoff can not exceed the upper bound of $d$'s true valuation. It is worth noting that the number of driver-MST pairs in $X(\widetilde{D}_t\backslash\{d\})$ is less than that in $X(\widetilde{D}_t)$ when the inequality $\Psi^{\prime}(\widetilde{D}_t\backslash\{d\}) < \Psi^{\prime}(\widetilde{D}_t)$ holds. This is because the task $k$ can not be assigned when excluding the driver $d$ from $\widetilde{D}_t$. There are two cases leading to this result, i.e., (1) driver $d$ is the only eligible bidder for the task $k$; (2) all driver-MST pairs to the task $k$ are prohibited except for the driver $d$ when solving the model [MST-P2]. In this situation, we set the upper bound of the driver's true valuation as the payoff.
	
	The whole process is summarized in Algorithm \ref{algorithm3}.
	\begin{algorithm}[ht]
		\caption{MST assignment with the RBC-MST mechanism}\label{algorithm3}
		\KwIn{The driver set $\mathcal{D}_t$, the MST set $\mathcal{K}_{T}$, the rider set $\mathcal{R}_t$, the true valuation matrix $V$, the budget $\Omega_T$}
		\KwOut{Trip matching result $X_{\mathscr{D}\mathscr{R}}$, task assignment result $X_{\mathscr{W}\mathscr{K}}$, task payment $\textbf{p}$}
		$\mathscr{N}\mathscr{K}_t, X_{\mathscr{W}\mathscr{K}}, X_{\mathscr{W}}, X_{\mathscr{K}}\leftarrow \varnothing$\;
		$flag \leftarrow \textbf{True}$\;
		\While{flag}
		{Solve the [MST-P2] without considering budget constraint to obtain the optimal solution $\mathscr{W}\mathscr{K}_t$, the set of winners $\mathscr{W}_t$ and the set of selected task requests $\mathscr{K}_t$\;
			\eIf{$\sum_{(d,k)\in\mathscr{W}\mathscr{K}_t}(C+\overline{b}l_{dk}+s(\overline{b},l_{dk})) > \Omega_T$}
			{$d^{*}, k^{*} = \argmax_{(d,k)\in \mathscr{W}\mathscr{K}_t}\left\{v_{dk}^{t} \right\}$\;
				\If{$\forall (d, k^{*})\in \mathscr{N}\mathscr{K}_t \enspace\textnormal{and}\enspace v_{dk^{*}}^{t} < +\infty$}
				    {$\tau_{k^{*}}^{t} \leftarrow 0$\;}
				$\mathscr{N}\mathscr{K}_t\leftarrow \mathscr{N}\mathscr{K}_t \cup \left\{(d^{*}, k^{*}) \right\}$\;}
			{$X_{\mathscr{W}\mathscr{K}}\leftarrow \mathscr{W}\mathscr{K}_t$\;
				$X_{\mathscr{W}}\leftarrow \mathcal{W}_t$\;
				$X_{\mathscr{K}}\leftarrow \mathcal{K}_t$\;
				$flag \leftarrow \textbf{False}$\;
			}
		}
		Find the optimal solution $X_{\mathscr{D}\mathscr{R}}$ following the steps 3 to 15 of Algorithm \ref{algorithm2}\;
		\For{$(d,k)\in \mathscr{W}\mathscr{K}_t$}
		{
			Re-calibrate the MST assignment model to obtain the optimal objective function value $\Psi^{\prime}(\widetilde{D}_t\backslash\{d\}) $ without considering the driver $d$\;
			Calculate the driver $d$'s payoff based on the tailored payment rule\;
		}
		\textbf{Return} \quad $X_{\mathscr{D}\mathscr{R}}, X_{\mathscr{W}\mathscr{K}}, \textbf{p}$
	\end{algorithm}
	
	\subsection{Economic properties}
	The proposed RBC-MST mechanism is proved to satisfy BB, IC, and IR properties. Additionally, under the condition of prohibiting driver-MST pairs in the tabu list $\mathscr{N}\mathscr{K}_t$, RBC-MST mechanism achieves the efficient allocation. Before proving the BB property, we show the assignment mechanism satisfies Lemma \ref{lemma_bidub}.
	\begin{remark}\label{remark_bidub}
		Let the upper bound of the driver $d$'s true valuation to be $\overline{v}_{dk}^{t} = C+\overline{b}l_{dk}+s(\overline{b}, l_{dk})$. When $s(\mathscr{B}_{d}^{t}, l_{dk})=0$, $\overline{v}_{dk}^{t}$ bounds $v_{dk}^{t}$ from above.
	\end{remark}
	\begin{lemma}\label{lemma_bidub}
		For any valid bid $\mathscr{B}_{d}^{t}\in \left[\underline{b}, \overline{b}\right]$, $\overline{v}_{dk}^{t} \geq v_{dk}^{t}$ if $s(\mathscr{B}_{d}^{t}, l_{dk})\neq 0$.
	\end{lemma}
	\begin{proof}
		By Equation \eqref{compfunc}, we have $\mathscr{B}_{d}^{t} < \alpha + \mu$ because  $s(\mathscr{B}_{d}^{t}, l_{dk})\neq 0$. Two scenarios arise concerning the relationship between $\alpha + \mu$ and $\overline{b}$:
		
		(1) $\overline{b} < \alpha + \mu$. $l_{dk} > \frac{C+\mu l_0}{\alpha+\mu-\mathscr{B}_{d}^{t}}$ and $\frac{C+\mu l_0}{\alpha+\mu-\overline{b}} > \frac{C+\mu l_0}{\alpha+\mu-\mathscr{B}_{d}^{t}}$ hold. If $s(\overline{b},l_{dk})\neq 0$, we have $l_{dk} > \frac{C+\mu l_0}{\alpha+\mu-\overline{b}}$ and hence $\overline{v}_{dk}^{t} = v_{dk}^{t} = C+\alpha l_{dk}+\mu(l_{dk}-l_0)$. If $s(\overline{b},l_{dk})=0$, we have $ \frac{C+\mu l_0}{\alpha+\mu-\mathscr{B}_{d}^{t}}< l_{dk} \leq \frac{C+\mu l_0}{\alpha+\mu-\overline{b}}$. Therefore, $\overline{v}_{dk}^{t} = C+\overline{b}l_{dk}\geq v_{dk}^{t} = C+\alpha l_{dk}+\mu(l_{dk}-l_0)$. 
		
		(2) $\alpha + \mu \leq \overline{b}$. $s(\overline{b},l_{dk})$ should be equal to 0. In this case, $\overline{v}_{dk}^{t} = C+\overline{b}l_{dk}\geq v_{dk}^{t} = C+\alpha l_{dk}+\mu(l_{dk}-l_0)$.
		
		In both cases, $\overline{v}_{dk}^{t} \geq v_{dk}^{t}$ given $s(\mathscr{B}_{d}^{t}, l_{dk})\neq 0$. This completes the proof.
	\end{proof}
	\begin{proposition}{\bf Budget Balance}
		In each plan cycle, the actual expenditure incurred by the RBC-MST mechanism is either less than or equal to the prescribed budget $\Omega_T$.
	\end{proposition}
	\begin{proof}
		Let the driver $d$'s true valuation to be $v_{dk}^{t}=C+\mathscr{B}_{d}^{t}l_{dk}+s(\mathscr{B}_d^{t},l_{dk})$ at the optimal solution $\mathscr{W}\mathscr{K}_t$. When both $s(\mathscr{B}_d^{t},l_{dk})$ and $s(\overline{b}, l_{dk})$ are equal to zero, $\overline{v}_{dk}^{t}\geq v_{dk}^{t}$ must hold by Equation \eqref{bid_lbub}. Per Lemma \ref{lemma_bidub}, $\overline{v}_{dk}^{t}\geq v_{dk}^{t}$ also holds if $s(\mathscr{B}_d^{t},l_{dk}) \neq 0$. Therefore $\overline{v}_{dk}^{t}$ is the upper bound of the driver $d$'s true valuation. 
		
		By Equation \eqref{ta2payoff}, the actual payments satisfies:
		\begin{equation*}
			\sum_{(d,k)\in\mathscr{W}\mathscr{K}_t}p_d^{t} \leq \sum_{(d,k)\in\mathscr{W}\mathscr{K}_t}(C+\overline{b}l_{dk}+s(\overline{b}, l_{dk})).
		\end{equation*}
		
		Due to the constraints in [MST-P2], the true valuation of the selected winners satisfies $\sum_{(d,k)\in\mathscr{W}\mathscr{K}_t}(C+\overline{b}l_{dk}+s(\overline{b}, l_{dk}))$. Thus, we obtain
		\begin{equation*}
			\sum_{(d,k)\in\mathscr{W}\mathscr{K}_t}p_d^{t} \leq \sum_{(d,k)\in\mathscr{W}\mathscr{K}_t}(C+\overline{b}l_{dk}+s(\overline{b}, l_{dk})) \leq \Omega_T.
		\end{equation*}
		
		Thus, the actual payment in any round of auction $T$ is less than or equal to the budget $\Omega_T$.
	\end{proof}
	
	\begin{proposition}{\bf Individual Rationality}
		With the RBC-MST mechanism, every eligible driver who submits bids will not incur negative utility.
	\end{proposition}
	\begin{proof}
		Let $\Psi_{\mathscr{W}\mathscr{K}_t}$ to be the corresponding value of the objective function of [MST-P2]. For any driver-MST pair $(d,k)\in \mathscr{W}\mathscr{K}_t$, we have $x_{dk}^{t}=1$. Moreover, we let $\overline{v}_{dk}^{t} = C+ \overline{b}l_{dk}+s(\overline{b}, l_{dk})$.

		When $\Psi^{\prime}(\widetilde{D}_t\backslash\{d\}) \geq \Psi^{\prime}(\widetilde{D}_t)$, the utility of driver $d$ satisfies
		\begin{align*}
			u_d^{t} &= p_d^{t} - v_{dk}^{t} = \min \bigg\{ v_{dk}^{t} + \big(\Psi^{\prime}(\widetilde{D}_t\backslash\{d\}) - \Psi^{\prime}(\widetilde{D}_t)\big), \enspace \overline{v}_{dk}^{t} \bigg\} - v_{dk}^{t} \\
			& = \min \bigg\{ \big(\Psi^{\prime}(\widetilde{D}_t\backslash\{d\}) - \Psi^{\prime}(\widetilde{D}_t)\big), \enspace \overline{v}_{dk}^{t} - v_{dk}^{t} \bigg\} > 0.
		\end{align*} This a direct result of Lemma \ref{lemma_bidub}.

		When $\Psi^{\prime}(\widetilde{D}_t\backslash\{d\}) < \Psi^{\prime}(\widetilde{D}_t)$, we have:
		\begin{align*}
			p_d^{t} - v_{dk}^{t} = \overline{v}_{dk}^{t} - v_{dk}^{t} \geq 0.
		\end{align*} by Lemma \ref{lemma_bidub}.
		
		Therefore, each selected driver can attain non-negative utility, while the utility of any deselected driver is null. This completes the proof of the IR property.
	\end{proof}
	
	\begin{proposition}{\bf (\textit{Ex-post}) Incentive Compatibility}
		Submitting a truthful bid is a dominant strategy for any driver in the RBC-MST mechanism given the other drivers submit truthful bids.
	\end{proposition}
	\begin{proof}
		We may discuss a driver change in utility when submitting an untruthful bid. A driver may be (1) selected, or (2) deselected for any MSTs when submits a bid truthfully. 
		
		By Equation \ref{ta2payoff}, there are three cases of the payment to any selected driver. The utility in submitting an untruthful bid may also evaluated accordingly.
		
		(1) $\Psi^{\prime}(\widetilde{D}_t\backslash\{d\}) < \Psi^{\prime}(\widetilde{D}_t)$. This case occurs only when the driver $d$ is the only one eligible bidder for $k$. When $d$ is excluded, an MST would also be left unassigned. In this case, the driver $d$ will obtain the upper bound of the own true valuation despite the bid. Thus, driver $d$'s utility remains the same if the bids satisfy the budget constraint and will be zero otherwise.
		
		(2) $\Psi^{\prime}(\widetilde{D}_t\backslash\{d\}) \geq \Psi^{\prime}(\widetilde{D}_t)$ and $\Psi^{\prime}(\widetilde{D}_t\backslash\{d\}) - \Psi^{\prime}(\widetilde{D}_t) > \overline{v}_{dk}^{t} - v_{dk}^{t}$. It means that the marginal benefit of the driver $d$ is above the valuation the upper bound of valuation minus true valuation. The driver $d$ who is selected will obtain the upper bound of the own true valuation. In this case, the utility of driver $d$ remains the same if the bids satisfy the budget constraint and will be zero otherwise.
		
		(3) $\Psi^{\prime}(\widetilde{D}_t\backslash\{d\}) \geq \Psi^{\prime}(\widetilde{D}_t)$ and $\Psi^{\prime}(\widetilde{D}_t\backslash\{d\}) - \Psi^{\prime}(\widetilde{D}_t) \leq \overline{v}_{dk}^{t} - v_{dk}^{t}$. Assuming that the driver $d$ can obtain higher payoff $\hat{p}_d^{t}$ by untruthful submission. Then, we have:
		\begin{equation*}
			\hat{p}_d^{t} - v_{dk}^{t} > p_d^{t} - v_{dk}^{t} = \Psi^{\prime}(\widetilde{D}_t\backslash\{d\}) - \Psi^{\prime}(\widetilde{D}_t)
		\end{equation*}
		
		Substituting $\hat{p}_d^{t} = \hat{v}_{dk}^{t} + \big(\Psi^{\prime}(\widetilde{D}_t\backslash\{d\}) - \hat{\Psi}^{\prime}(\widetilde{D}_t)\big)$ into above inequality, we get:
		\begin{align*}
			\quad &\hat{v}_{dk}^{t} + \big(\Psi^{\prime}(\widetilde{D}_t\backslash\{d\}) - \hat{\Psi}^{\prime}(\widetilde{D}_t)\big) - v_{dk}^{t} > \Psi^{\prime}(\widetilde{D}_t\backslash\{d\}) - \Psi^{\prime}(\widetilde{D}_t) \\
			\Rightarrow\quad & \hat{v}_{dk}^{t} - \hat{\Psi}^{\prime}(\widetilde{D}_t) - v_{dk}^{t} > - \Psi^{\prime}(\widetilde{D}_t) \\
			\Rightarrow\quad &\hat{\Psi}^{\prime}(\widetilde{D}_t) + v_{dk}^{t} - \hat{v}_{dk}^{t} < \Psi^{\prime}(\widetilde{D}_t) \\
			\Leftrightarrow\quad &\hat{v}_{dk}^{t} + \sum_{(l,j)\in\mathscr{WK}_t}v_{lj}^{t}x_{lj}^{t} + v_{dk}^{t} - \hat{v}_{dk}^{t} < \Psi^{\prime}(\widetilde{D}_t) \\
			\Rightarrow\quad &v_{dk}^{t} + \sum_{(l,j)\in\mathscr{WK}_t}v_{lj}^{t}x_{lj}^{t} < \Psi^{\prime}(\widetilde{D}_t)
		\end{align*}
		
		It contradicts that $\Psi^{\prime}(\widetilde{D}_t) = v_{dk}^{t} + \sum_{(l,j)\in\mathscr{WK}_t}v_{lj}^{t}x_{lj}^{t}$ is optimal objective value. 
		
		Therefore, the utility of any selected driver $d$ can not increase by untruthful bidding. 
		
		We may now discuss the utility change of any deselected driver when submits an untruthful bid. There are two cases for these drivers. 
		
		(1) When the driver $d$ overbids, the valuations to all tasks increase. The optimal driver-MST pairs remain unchanged, which means that the driver $d$ is still deselected. In this case, the utility of driver $d$ remains zero.
		
		(2) When the driver $d$ underbids, the valuations to all tasks decrease. As the bid continuously decreases, $d$ may take over task $k$ from the initially assigned driver $l$, whose valuation to the task is $v_{lk}^{t}$. Let $v_{dk}^{t}, \hat{v}_{dk}^{t}$ denote driver $d$'s true valuation and false valuation to task $k$, respectively. When $\hat{v}_{dk}^{t} > v_{lk}^{t}$, the driver $d$ is still not selected and the utility remains the same. When $\hat{v}_{dk}^{t} \leq v_{lk}^{t}$, $d$ replaces $l$ as the selected driver, and we have $\hat{v}_{dk}^{t} \leq v_{lk}^{t} \leq v_{dk}^{t} \leq \overline{v}_{dk}^{t}$. Let the sum of the other driver-MST pairs' true valuations be $\varPsi_{-k}$. Then $d$'s payoff could be calculated by Equation \eqref{ta2payoff}:
		\begin{align*}
			\hat{p}_d &= \min\big\{\hat{v}_{dk}^{t} + \big(\hat{\Psi^{\prime}}((\widetilde{D}_t\backslash\{d\})) - \hat{\Psi^{\prime}}(\widetilde{D}_t)\big), \overline{v}_{dk}^{t} \big\} \\
			&= \min\big\{\hat{v}_{dk}^{t} + \big(v_{lk}^{t} + \varPsi_{-k} - (\hat{v}_{dk}^{t} + \varPsi_{-k}) \big), \overline{v}_{dk}^{t} \big\} = \min\big\{v_{lk}^{t}, \overline{v}_{dk}^{t}\big\} = v_{lk}^{t} \leq v_{dk}^{t}
		\end{align*}
		
		The payoff of driver $d$ is less than the own true valuation $v_{dk}^{t}$ when submitting untruthfully. In this case, the utility of driver $d$ is negative.
		
		Thus, submitting truthfully to the platform is a dominant strategy for either a selected or a deselected driver. This completes the proof.
	\end{proof}

	\section{Computational experiments}\label{section:numer}
	We conducted numerical tests using a large-scale instance generated from New York City taxi data. Initially, we assessed the performance of both the VCG-MST and RBC-MST mechanisms across various scenarios. Following this, we examined how taxi fleet size and the ratio between the two types of taxis impact the completion rate of sensing tasks. Lastly, we investigated the average payoffs for both Type-A and Type-B drivers under different bidding bounds, aiming to determine the willingness of Type-B drivers to undertake MSTs. 
	
	All algorithms presented in this article were implemented in Python 3.10, and integer programming problems were solved using Gurobi 10.0.0.
	\subsection{Description of experimental data and parameter settings}
	We chose the Manhattan borough of New York City as the testing area for our numerical experiments. Taking advantage of yellow taxi trip records from January 3rd to January 6th, 2022, sourced from the NYC Taxi and Limousine Commission (TLC) \citep{TLC2022}, we inferred estimated cruising times for calculating opportunity costs in the trip matching process. To assess the performance of the proposed mechanisms, we employed trip records from January 7th, 2022, spanning two hours.
	
	To simplify calculations, we converted the area IDs of pick-up and drop-off locations for each trip record into network node IDs corresponding to the respective area. To test the performance of mechanisms under different trip demand scenarios, we extracted low trip demand scenarios (1,000 trip requests during the time horizon) and high trip demand scenarios (2,000 trip requests during the time horizon) from the raw data using standard sampling techniques. It is worth noting that the extracted trip requests is a subset of the actual trip requests. By which, we computed the average origin-destination (OD) distance for all trip requests. The value is around 3km. Thus, we set $l_0\approx 3$ km.
	
	The overall time horizon was set to 2 hours. Type-A taxi operations were scheduled every 30 seconds, while Type-B taxi planning cycles were set at 5 minutes. Each cycle was divided into 3 phases as described in Section \ref{section:model}: ride-hailing task matching $T_1$, mobile sensing task bidding $T_2$, and mobile sensing task assignment $T_3$. The durations of these phases were 3 minutes, 1.5 minutes, and 30 seconds, respectively.
	
	We randomly selected the locations of 80 MSTs within the study area. Figure \ref{fig:LH_tripdemand} displays the heat map of trip demand scenarios and the locations of MSTs within the Manhattan road network. Additionally, we randomly placed the depot of dedicated vehicles within the area to facilitate distance calculations between the dedicated vehicles and the MSTs. Refer to Table \ref{tab:parameter_settings} for parameter values.
	\begin{figure}[htbp]
		\centering
		\includegraphics[width=\textwidth]{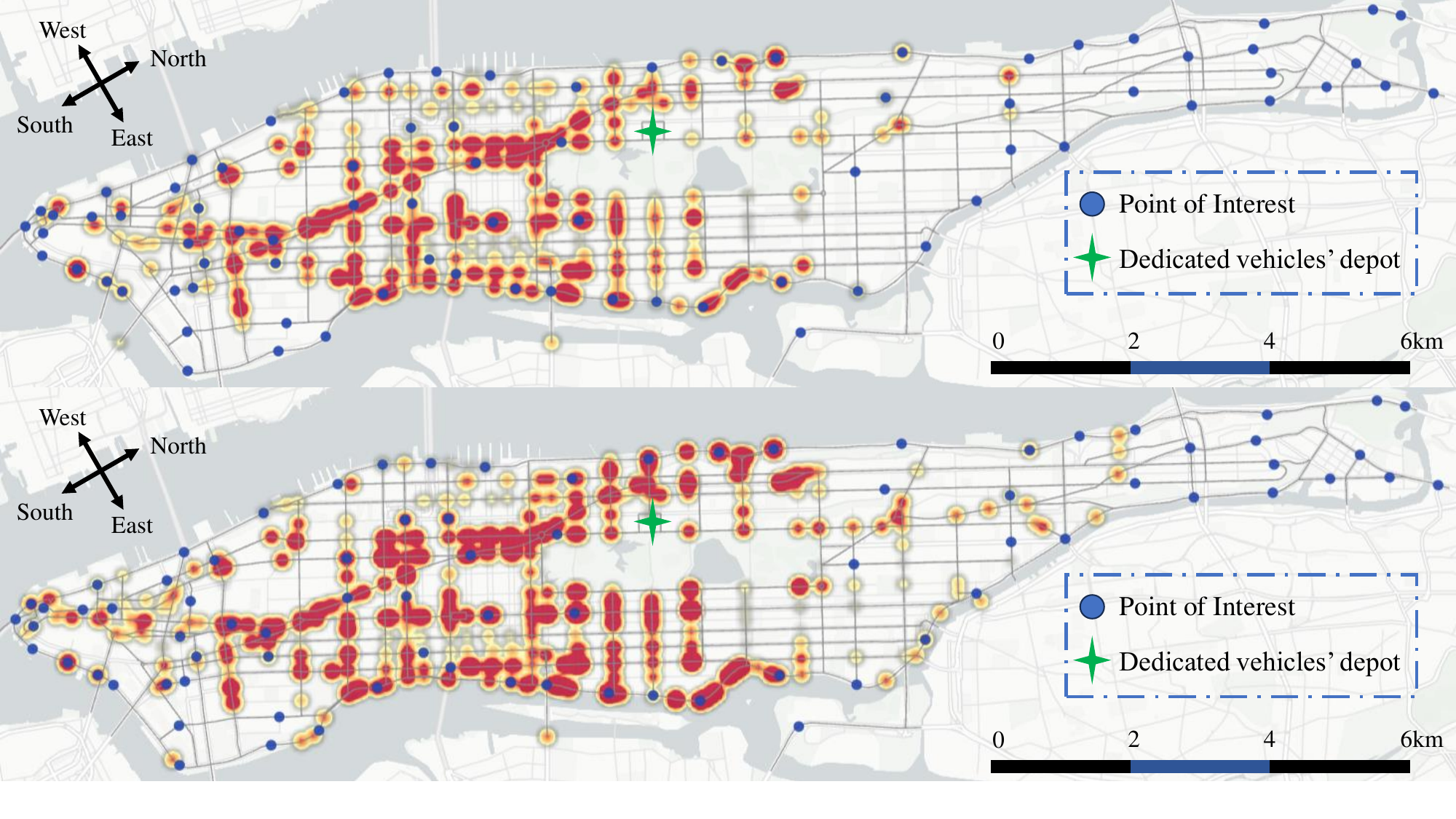}
		\caption{\centering{The low trip demand scenario, the high trip demand scenario and the locations of selected PoIs.}}
		\label{fig:LH_tripdemand}
	\end{figure}
	\begin{table}[ht]
		\centering
		\scriptsize
		\caption{Values of parameters}
		\begin{tabularx}{\textwidth}{p{0.2\textwidth}X p{0.2\textwidth}X}
			\toprule[2pt]
			Parameters & Values & Parameters & Values \\
			\midrule
			$p_r^{s}$ & 12 units & $\alpha$ & 2 units/km \\
			$\beta_1$ & 1.70 units/km & $\underline{b}$ & 2 units/km \\
			$\beta_2$ & 0.50 units/min & $\overline{b}$ & 4 units/km \\
			$L_0^{s}$ & 3km & $\mu$ & 1 unit/km \\
			$t_0^{s}$ & 10min & $c_q$ & 8 units/km \\
			$L_{ub}$ & 2km & $C$ & 15 units \\
			$\overline{V}$ & 35km/h & $l_0$ & 3km \\
			\bottomrule[2pt]
		\end{tabularx}
		\label{tab:parameter_settings}
	\end{table}
	
	\subsection{Performance the VCG-MST and RBC-MST under different travel demands}\label{subsection:experiment1}
	We begin by evaluating the performance of the two mechanisms under different levels of travel demand. The total number of vehicles remains fixed at 140, while the ratio between the two types of vehicles varies across the tests. After executing the coordination strategy, three key indices are calculated: the social surplus, the remaining budget, and the completion rate of MSTs by taxis.
	
	The social surplus represents the difference between the fixed cost of using dedicated vehicles to complete tasks and the total expense incurred by the ride-hailing platform through the two MST assignment mechanisms. Essentially, it reflects the combined interests of the data user and the ride-hailing platform. The social surplus can be calculated as follows:
	\begin{equation}\label{social_welfare}
		SS = \sum_{k\in \mathcal{K}}c_q l_{qk} - \sum_{T\in\mathcal{T}}\sum_{(d,k)\in \mathscr{W}\mathscr{K}_T}p_{dk}^{T},
	\end{equation} where the first term on the right-hand side represents the fixed cost of utilizing dedicated vehicles to complete all MSTs, while the second term denotes the total expenses incurred by the ride-hailing platform for hailing tasks.
	
	The remaining budget quantifies the deficit or revenue of the ride-hailing platform in the MST business. Let $RB$ represent the remaining budget at the end of the time horizon, and $B_{init}$ denote the initial budget provided in advance. $RB$ can be calculated as follows:
	\begin{equation}\label{remaining_budget}
		RB = B_{init} - \sum_{T\in\mathcal{T}}\sum_{(d,k)\in \mathscr{W}\mathscr{K}_T}p_{dk}^{T}.
	\end{equation}
	
	The completion rate of MSTs by taxis gauges the appeal of taxi-based mobile sensing to the data user, indicating the extent to which the taxi fleet can substitute dedicated vehicles. Let $CR$ represent the completion rate of MSTs, $\lvert \mathscr{WK}_T\rvert$ denote the number of task orders assigned to drivers in planning cycle $T$, and $\lvert\mathcal{K}\rvert$ denote the total number of MSTs. $CR$ can be calculated as follows:
	
	\begin{equation}\label{completion_rate}
		CR = \frac{\sum_{T\in\mathcal{T}}\lvert \mathscr{W}\mathscr{K}_T\rvert}{\lvert\mathcal{K}\rvert}.
	\end{equation}
	
	Due to the random nature of the inputs, we conducted 5 repeated experiments for both the low-demand and high-demand cases, calculating the average of each index. Figure \ref{fig:Ltd_3subfigure} compares the performance of the two MST assignment mechanisms under the low trip demand scenario. It's notable that the RBC-MST mechanism outperforms the VCG-MST mechanism in terms of social welfare ($SW$) and remaining budget ($RB$), but performs slightly worse than the VCG-MST mechanism in completion rate ($CR$). This discrepancy arises from the allocation rules of the two mechanisms.
	
	The RBC-MST mechanism regulates the number of MSTs commissioned in each planning cycle based on the allocated budget. Consequently, some tasks, which might not be economically viable for the ride-hailing platform upon initial release, are deferred and assigned to drivers in subsequent auctions. This approach enables the platform to generate more social surplus and maintain a higher remaining budget. On the contrary, the VCG-MST mechanism assigns MSTs to drivers as long as their contributions to the objective are positive, following a more greedy approach. While this results in a higher task completion rate, it may lead to excessive expenses on specific MSTs, potentially causing a deficit.
	
	From Figure \ref{fig:Ltd_3subfigure}, we observe that the RBC-MST mechanism operates optimally when the taxi fleet consists of 20 Type-B vehicles and 120 Type-A vehicles in this low trip demand scenario. Under these conditions, both the social surplus and completion rate reach their maximum, and the remaining budget remains non-negative.
	
	\begin{figure}[htbp]
		\centering
		\includegraphics[width=\textwidth]{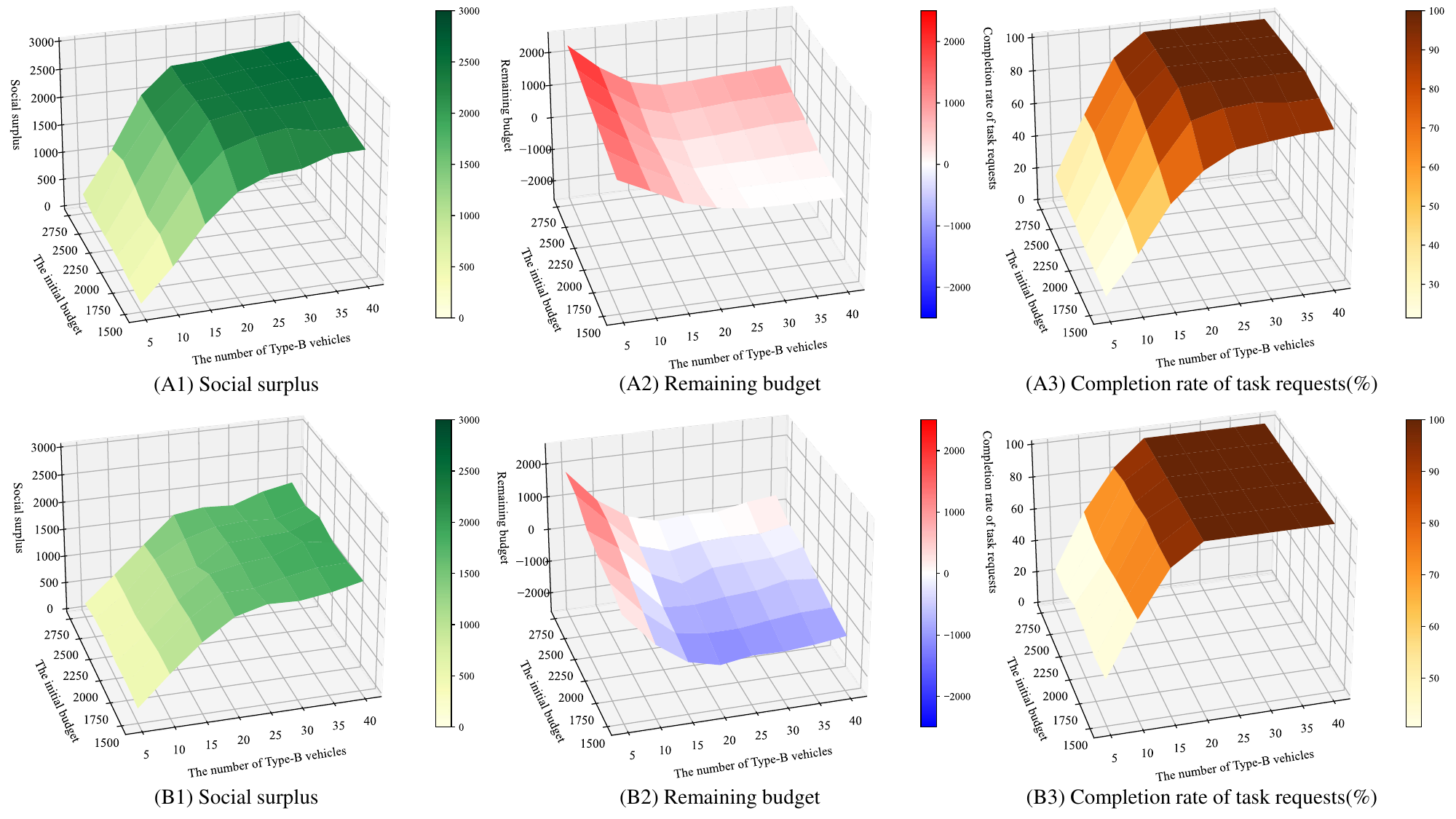}
		\caption{\centering{Comparison of (A) RBC-MST and (B) VCG-MST mechanism under the low trip demand scenario}}
		\label{fig:Ltd_3subfigure}
	\end{figure}
	
	Similar numerical experiments were conducted for the high trip demand scenario, and the results of the RBC-MST and VCG-MST mechanisms are summarized in Figure \ref{fig:Htd_3subfigure}. Notably, there are no significant differences between these two mechanisms across all three indices, although the RBC-MST mechanism performs slightly better in terms of social surplus and remaining budget, but slightly worse in completion rate. This observation can be attributed to the fact that when travel demand is high, most Type-A and Type-B vehicles are allocated to hailing orders rather than MSTs.
	\begin{figure}[htbp]
		\centering
		\includegraphics[width=\textwidth]{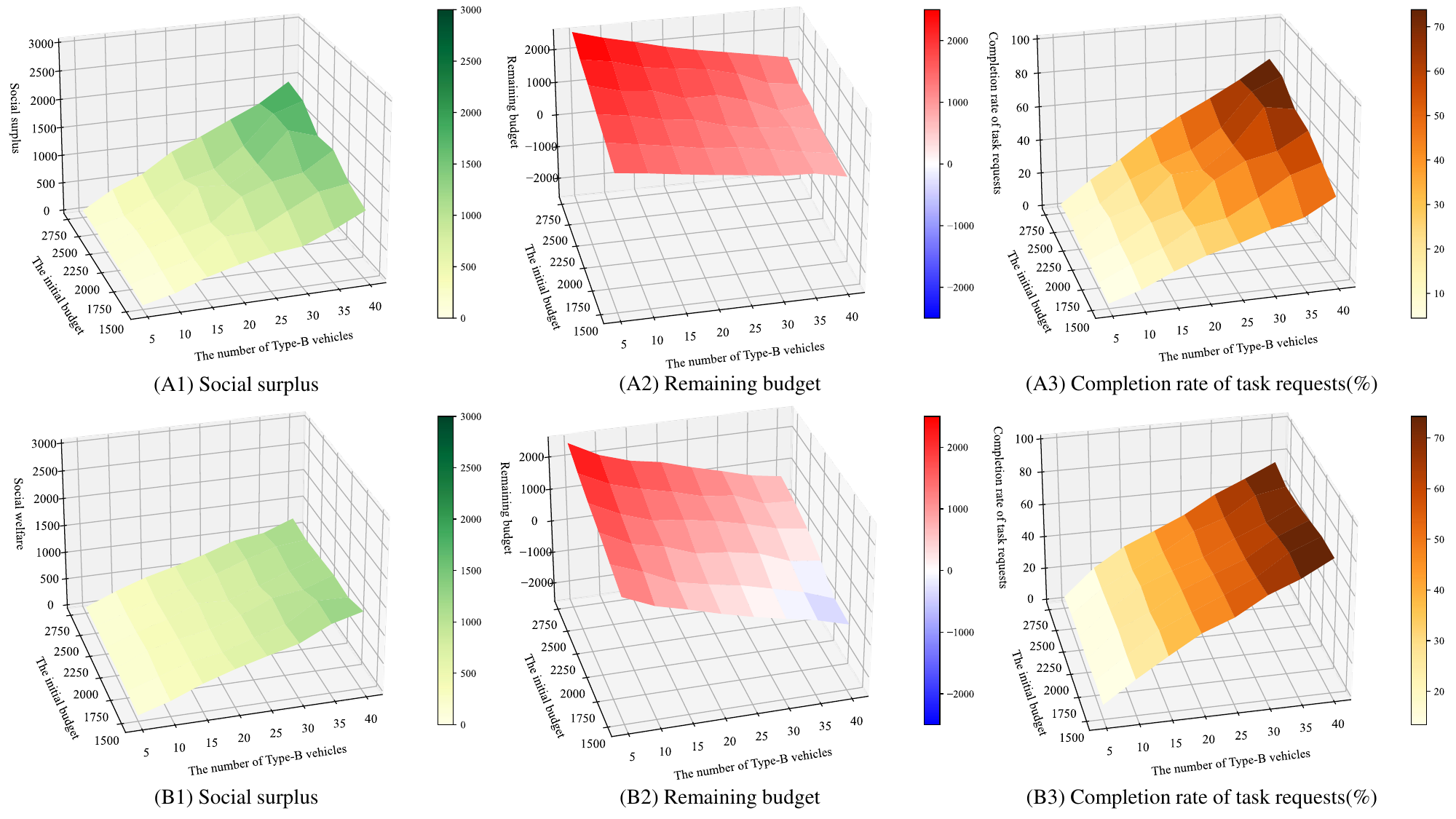}
		\caption{\centering{Comparison of (A) RBC-MST and (B) VCG-MST mechanism under the high trip demand scenario}}
		\label{fig:Htd_3subfigure}
	\end{figure}
	
	\subsection{Impact of the vehicle fleet sizes} \label{subsection:experiment2}
	In this subsection, we investigate the impact of taxi fleet sizes on the performances of the two mechanisms. We measure the service level for riders by calculating the average trip matching rate and average waiting time. Additionally, we evaluate the effectiveness of the coordinated strategy for mobile sensing using the social surplus and the completion rate of MSTs. For these experiments, we employ 20 Type-B vehicles in the taxi network, while the number of Type-A vehicles range from 40 to 160 with a step size of 20. Each experiment is repeated 5 times, and the average value of each index is calculated and presented in Figure \ref{fig:trip_kpi}. It's observed that the trip matching rate generally increases and the average waiting time of riders decreases with the increase in the number of Type-A vehicles. Furthermore, both mechanisms provide great service level of trip matching as the size of vehicle fleet increases.
	
	\begin{figure}[htbp]
		\centering
		\includegraphics[width=0.95\textwidth]{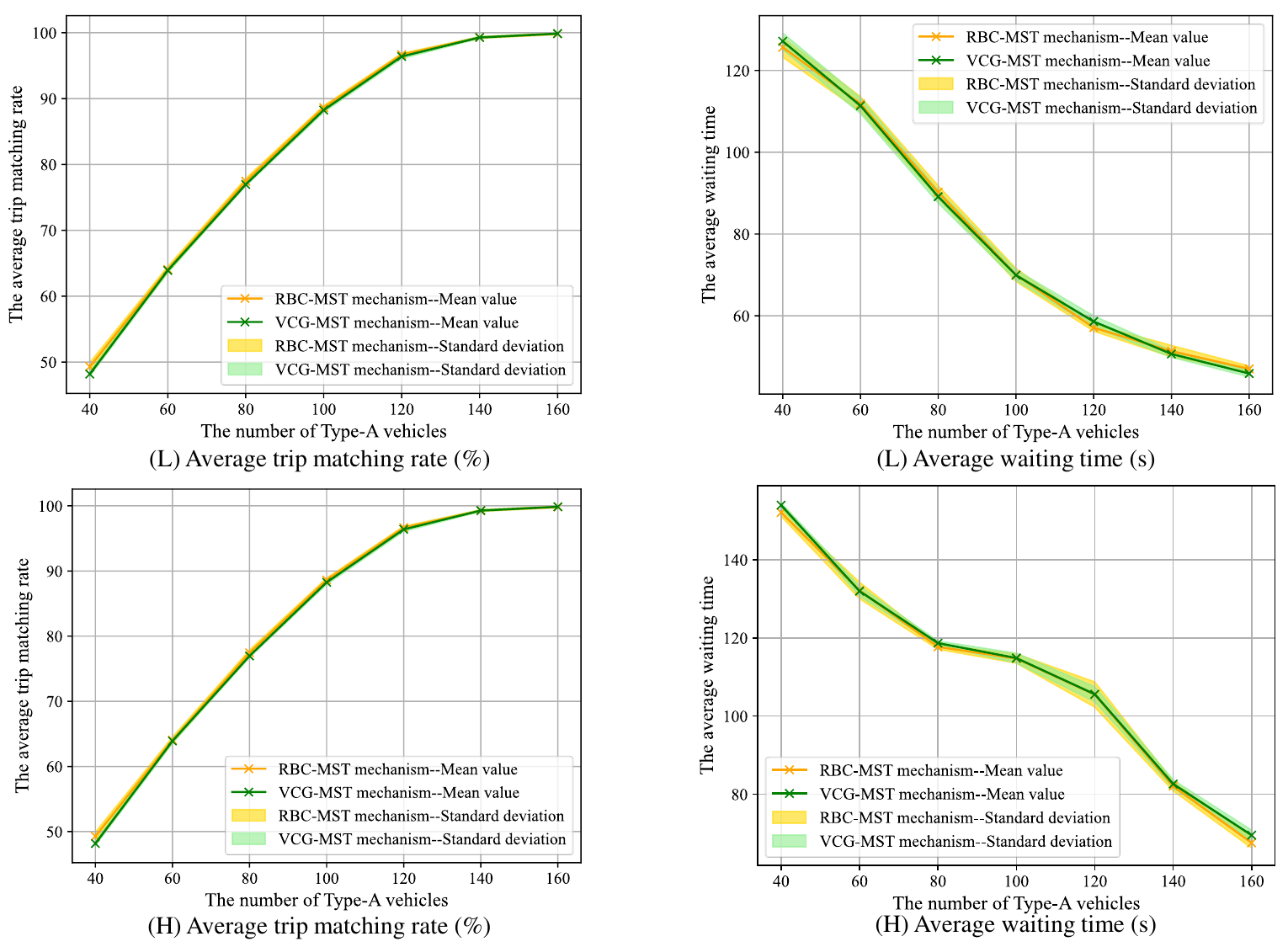}
		\caption{\centering{The impact of the various number of Type-A vehicles on trip matching in (L) low travel demand scenario and (H) high travel demand scenario}}
		\label{fig:trip_kpi}
	\end{figure}
	
	Figure \ref{fig:task_kpi} illustrates the impact of the number of Type-A vehicles on the performances of the MST assignment. In scenarios with low travel demand, the social surplus increases with the number of available vehicles for both MST assignment mechanisms. However, this increase plateaus when the total number of vehicles reaches around 120, indicating that all sensing tasks could be efficiently commissioned to the Type-B vehicles. Notably, the RBC-MST mechanism consistently yields a large social surplus across all tests, making it particularly attractive to both the data user and the ride-hailing platform.
	
	While a similar increasing pattern of the social surplus is observed in scenarios with high travel demand, no mechanism demonstrates consistently superior performance across all tests. The VCG-MST mechanism achieves a higher completion rate and significant social surplus when the number of Type-A vehicles is around 100. However, this advantage diminishes as the number of vehicles exceeds 120.
	
	\begin{figure}[htbp]
		\centering
		\includegraphics[width=0.95\textwidth]{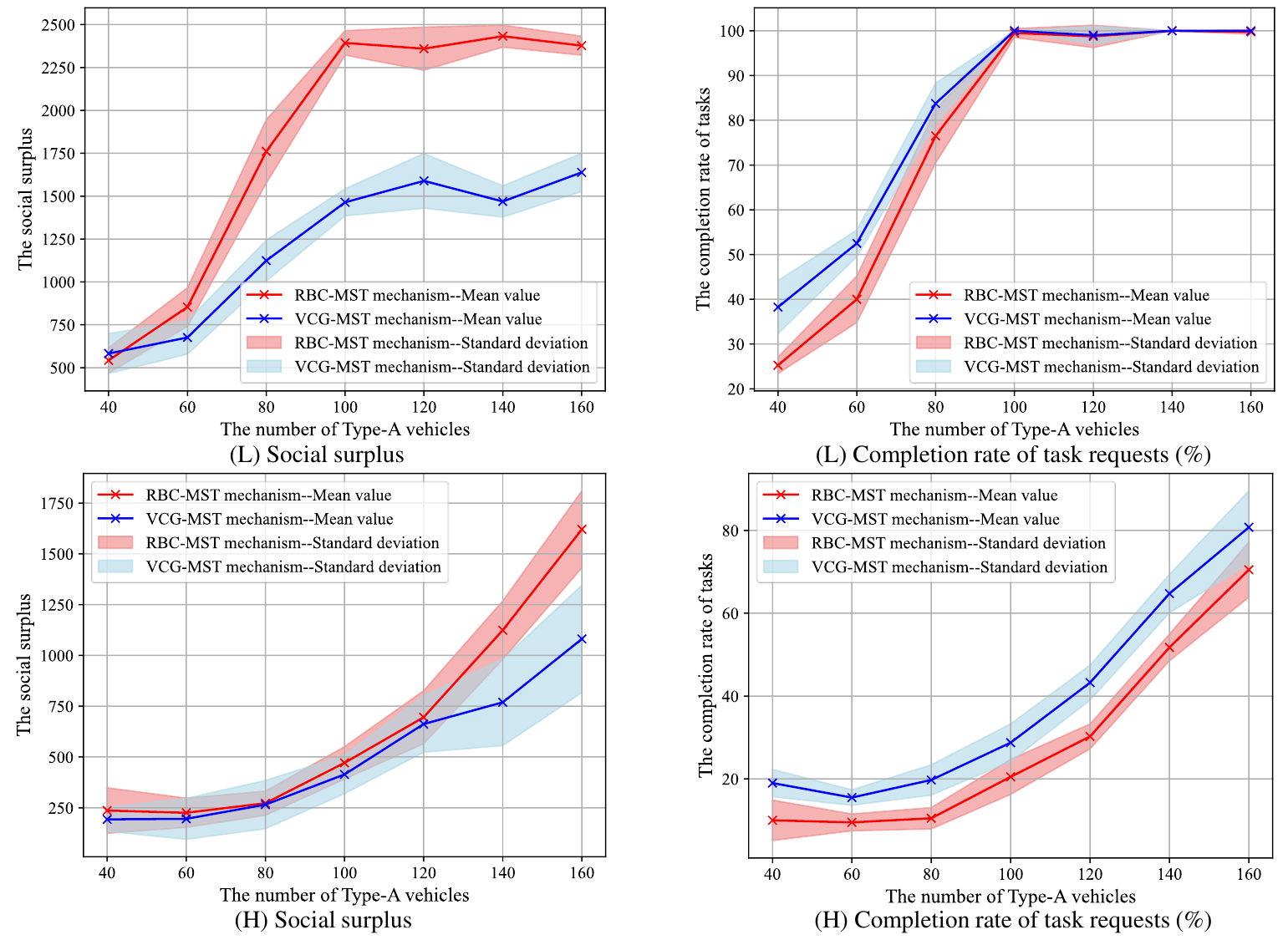}
		\caption{\centering{The impact of the various number of Type-A vehicles on MST assignment in (L) low travel demand scenario and (H) high travel demand scenario}}
		\label{fig:task_kpi}
	\end{figure}
	
	\subsection{Income of the Type-A and Type-B drivers}
	Given the potential for generating high social surplus in scenarios with low travel demand, we may first investigate the welfare of the drivers operating in the taxi-based mobile sensing business under these conditions. For this examination, we set the initial budget and fleet size to 2,000 and 140, respectively.
	
	We might also wonder how to incentivize drivers to accept MSTs. As outlined in Table \ref{tab:parameter_settings}, the bid value per unit distance for vehicles ranges from 2 to 7. The lower bound corresponds to the unit income of Type-A drivers, while the upper bound aligns with the unit cost of dedicated vehicles. We can further subdivide this range into smaller increments, such as [2,4], [3,5], [4,6], and [5,7]. Conducting 5 experiments for each range of bounds, we summarize the trip matching and MST assignment indices in Table \ref{tab:LTD_data_result}. With increasing bid bounds, the average payoff for Type-B drivers rises, though at the expense of social surplus and completion rate.
	
	\begin{table}[htbp]
		\centering
		\scriptsize
		\caption{The impact of performing sensing tasks under the low trip demand scenario}
		\begin{threeparttable}
			\begin{tabularx}{\textwidth}{>{\centering\arraybackslash}p{0.05\textwidth} >{\centering\arraybackslash}p{0.05\textwidth} >{\centering\arraybackslash}p{0.05\textwidth} >{\centering\arraybackslash}p{0.05\textwidth} >{\centering\arraybackslash}X >{\centering\arraybackslash}X >{\centering\arraybackslash}X >{\centering\arraybackslash}X >{\centering\arraybackslash}X >{\centering\arraybackslash}X}
				\toprule[2pt]
				\multirow{2}*{BD} & \multirow{2}*{$B_{init}$} & \multirow{2}*{$N_{AB}$} & \multirow{2}*{$N_B$} & \multicolumn{2}{c}{Trip matching} & \multicolumn{2}{c}{Task assignment} & \multirow{2}*{AP-A} & \multirow{2}*{AP-B} \\ \cline{5-8}
				\rule{0pt}{10pt} & & & & AWT (s) & ATR (\%) & ACR (\%) & ASW  & & \\
				\midrule
				\multirow{5}*{$U[2,4]$} & \multirow{5}*{2000} & \multirow{5}*{140} & 20 & 57.77 & 100 & 100 & 2356.17 & 87.24 & 143.05 \\
				& & & 25 & 58.40 & 100 & 100 & 2448.55 & 89.10 & 122.92 \\
				& & & 30 & 59.49 & 99.9 & 100 & 2453.94 & 90.55 & 111.59 \\
				& & & 35 & 61.67 & 100 & 100 & 2485.01 & 91.72 & 104.26 \\
				& & & 40 & 61.89 & 99.9 & 100 & 2494.86 & 92.84 & 99.01 \\
				\midrule
				\multirow{5}*{$U[3,5]$} & \multirow{5}*{2000} & \multirow{5}*{140} & 20 & 57.76 & 99.9 & 97.5 & 2205.85 & 87.61 & 143.95 \\
				& & & 25 & 58.52 & 100 & 98.6 & 2321.43 & 88.43 & 129.16 \\
				& & & 30 & 58.88 & 100 & 99.9 & 2379.44 & 89.87 & 116.10 \\
				& & & 35 & 60.24 & 99.9 & 99.9 & 2396.47 & 91.02 & 108.97 \\
				& & & 40 & 61.06 & 99.9 & 100 & 2420.39 & 92.15 & 103.24 \\
				\midrule
				\multirow{5}*{$U[4,6]$} & \multirow{5}*{2000} & \multirow{5}*{140} & 20 & 58.14 & 99.9 & 94.9 & 1972.08 & 86.73 & 149.21 \\
				& & & 25 & 57.44 & 100 & 96.2 & 2128.48 & 88.88 & 127.25 \\
				& & & 30 & 59.34 & 100 & 97.2 & 2168.21 & 89.06 & 120.18 \\
				& & & 35 & 59.41 & 100 & 98.7 & 2313.82 & 91.84 & 107.47 \\
				& & & 40 & 61.69 & 99.9 & 98.6 & 2291.80 & 92.38 & 104.36 \\
				\midrule
				\multirow{5}*{$U[5,7]$} & \multirow{5}*{2000} & \multirow{5}*{140} & 20 & 57.65 & 99.9 & 91.1 & 1777.37 & 86.84 & 148.58 \\
				& & & 25 & 57.86 & 99.9 & 92.7 & 2071.31 & 88.18 & 130.55 \\
				& & & 30 & 58.58 & 99.9 & 94.9 & 2046.30 & 89.06 & 120.30 \\
				& & & 35 & 60.01 & 99.9 & 96.7 & 2160.73 & 91.19 & 109.39 \\
				& & & 40 & 60.41 & 99.9 & 97.5 & 2217.19 & 91.57 & 105.99 \\
				\bottomrule[2pt]
			\end{tabularx}
			\begin{tablenotes}
				\item \scriptsize Notes:
				\item \scriptsize 1. BD: the bounds of bids, $N_{AB}$: the total number of vehicles, $N_B$: the number of Type-B vehicles.
				\item \scriptsize 2. AWT: average waiting time, ATR: average trip matching rate, ACR: average completion rate, ASW: average social welfare.
				\item \scriptsize 3. AP-A: average payoff of Type-A drivers, AP-B: average payoff of Type-B drivers.
			\end{tablenotes}
		\end{threeparttable}
		\label{tab:LTD_data_result}
	\end{table}
	
	Figure \ref{fig:LTD_driver_earnings} depicts the average earnings of Type-A and Type-B drivers across these experiments. When the count of Type-B vehicles is 30 or fewer, and the bid prices fall within the interval $[2,4]$, Type-B drivers stand to earn significantly more than their Type-A counterparts by undertaking monitoring tasks. Furthermore, the increment in the number of Type-B drivers doesn't markedly impact the earnings of Type-A drivers.
	\begin{figure}[htbp]
		\centering
		\includegraphics[width=\textwidth]{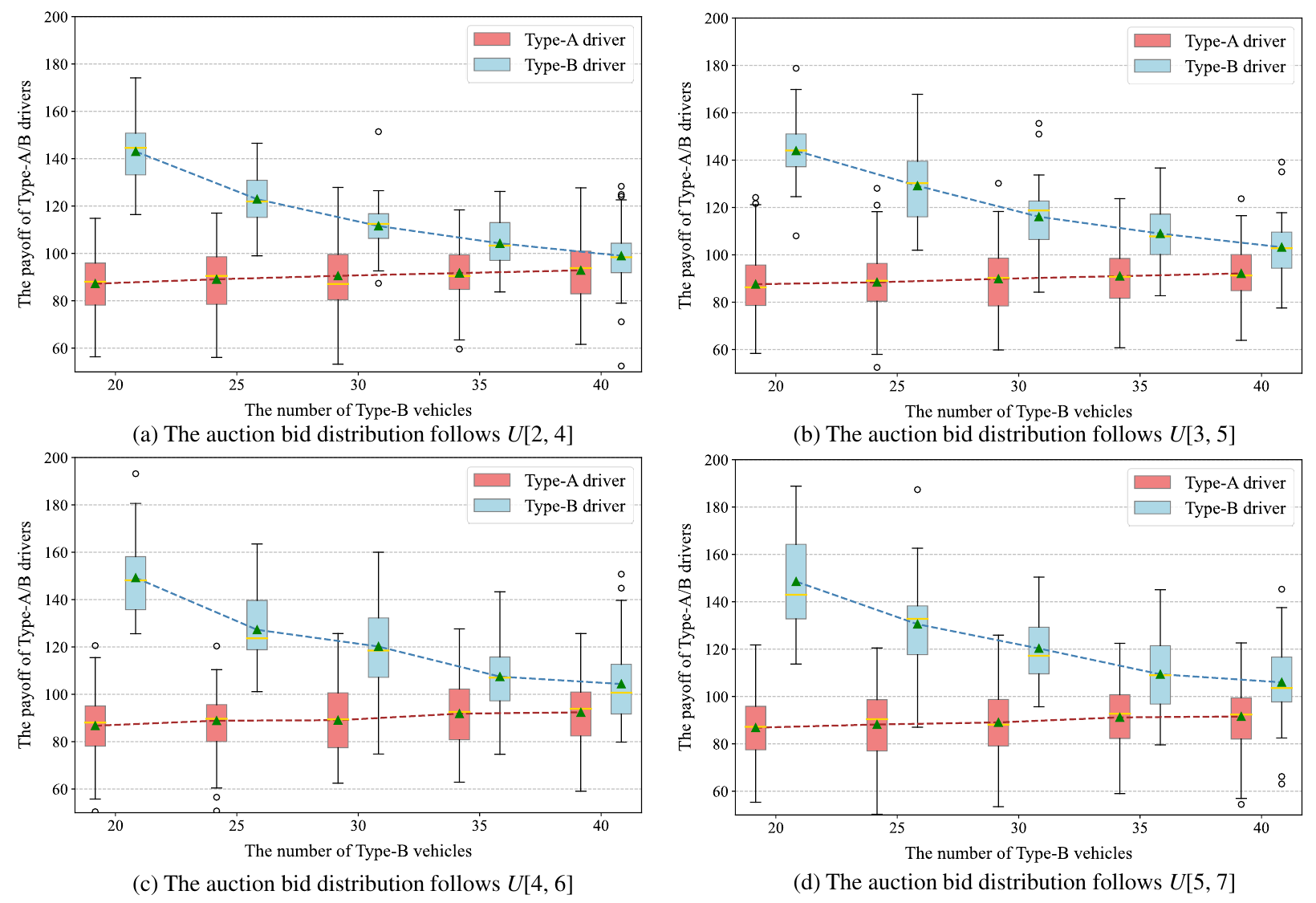}
		\caption{\centering{The average payoff trend of Type-A drivers and Type-B drivers under the low trip demand scenario (constant initial budget and fixed fleet size)}}
		\label{fig:LTD_driver_earnings}
	\end{figure}
	
	We conducted similar experiments in the high-demand scenario, and the performance indices for trip matching and MST assignment are presented in Table \ref{tab:HTD_data_result}. We generally observed similar trends in the performance indices as in the low-demand scenario when the parameters changed. However, in scenarios where the bidding bounds are too low, particularly within the range of $[2,4]$, Type-B drivers may incur losses, as depicted in Figure \ref{fig:HTD_driver_earnings}. This outcome is reasonable because drivers may have to travel away from areas with high ride-hailing demand to perform MSTs. While the unit payoff for performing sensing tasks may be higher, they may miss out on opportunities to serve more riders. Another interesting observation is that the average payoff of Type-B drivers can slightly increase with the rise in the number of Type-B drivers. This could be because Type-B drivers are more sparsely distributed in the area, requiring fewer of them to relocate to distant MSTs.
	
	\begin{figure}[htbp]
		\centering
		\includegraphics[width=\textwidth]{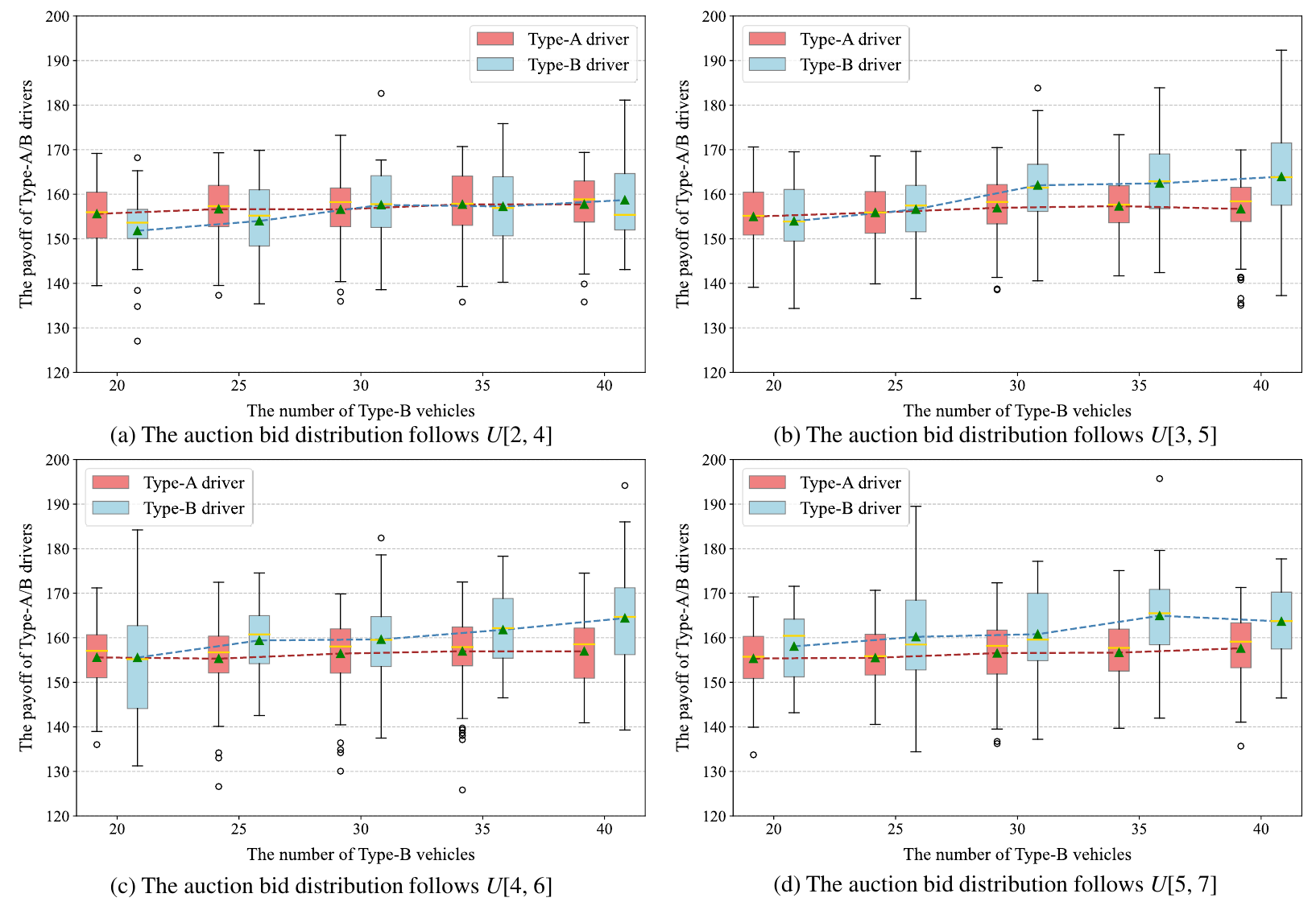}
		\caption{\centering{The average payoff trend of Type-A drivers and Type-B drivers under the high trip demand scenario (constant initial budget and fixed fleet size)}}
		\label{fig:HTD_driver_earnings}
	\end{figure}
	
	\begin{table}[htbp]
		\centering
		\scriptsize
		\caption{The impact of performing sensing tasks under the high trip demand scenario}
		\begin{threeparttable}
			\begin{tabularx}{\textwidth}{>{\centering\arraybackslash}p{0.05\textwidth} >{\centering\arraybackslash}p{0.05\textwidth} >{\centering\arraybackslash}p{0.05\textwidth} >{\centering\arraybackslash}p{0.05\textwidth} >{\centering\arraybackslash}X >{\centering\arraybackslash}X >{\centering\arraybackslash}X >{\centering\arraybackslash}X >{\centering\arraybackslash}X >{\centering\arraybackslash}X}
				\toprule[2pt]
				\multirow{2}*{BD} & \multirow{2}*{$B_{init}$} & \multirow{2}*{$N_{AB}$} & \multirow{2}*{$N_B$} & \multicolumn{2}{c}{Trip matching} & \multicolumn{2}{c}{Task assignment} & \multirow{2}*{AP-A} & \multirow{2}*{AP-B} \\ \cline{5-8}
				\rule{0pt}{10pt} & & & & AWT (s) & ATR (\%) & ACR (\%) & ASW & & \\
				\midrule
				\multirow{5}*{$U[2,4]$} & \multirow{5}*{2000} & \multirow{5}*{140} & 20 & 100.22 & 96.5 & 30 & 618.19 & 155.59 & 151.79 \\
				& & & 25 & 100.64 & 96.5 & 36.6 & 780.43 & 156.66 & 154.00 \\
				& & & 30 & 101.59 & 96.2 & 47.5 & 1147.17 & 156.60 & 157.61 \\
				& & & 35 & 101.50 & 96.0 & 55.8 & 1455.09 & 157.74 & 157.25 \\
				& & & 40 & 102.54 & 95.6 & 62 & 1602.93 & 157.74 & 158.70 \\
				\midrule
				\multirow{5}*{$U[3,5]$} & \multirow{5}*{2000} & \multirow{5}*{140} & 20 & 99.81 & 96.6 & 27.1 & 447.35 & 154.98 & 153.96 \\
				& & & 25 & 99.67 & 96.4 & 35.2 & 690.41 & 155.89 & 156.61 \\
				& & & 30 & 101.34 & 96.2 & 46.9 & 1169.04 & 156.95 & 162.02 \\
				& & & 35 & 102.09 & 96.1 & 55.4 & 1233.07 & 157.33 & 162.45 \\
				& & & 40 & 102.10 & 95.8 & 65.2 & 1415.69 & 156.70 & 163.94 \\
				\midrule
				\multirow{5}*{$U[4,6]$} & \multirow{5}*{2000} & \multirow{5}*{140} & 20 & 98.99 & 96.6 & 27.5 & 579.98 & 155.59 & 155.56 \\
				& & & 25 & 101.27 & 96.6 & 31.2 & 659.30 & 155.31 & 159.41 \\
				& & & 30 & 101.71 & 96.2 & 40.6 & 846.27 & 156.45 & 159.63 \\
				& & & 35 & 102.18 & 96.0 & 52.6 & 994.53 & 156.98 & 161.79 \\
				& & & 40 & 102.06 & 95.7 & 61.1 & 1285.43 & 156.98 & 164.41 \\
				\midrule
				\multirow{5}*{$U[5,7]$} & \multirow{5}*{2000} & \multirow{5}*{140} & 20 & 100.01 & 96.6 & 23.5 & 468.54 & 155.35 & 158.07 \\
				& & & 25 & 101.31 & 96.6 & 34.1 & 501.30 & 155.51 & 160.21 \\
				& & & 30 & 100.91 & 96.3 & 41.1 & 676.95 & 156.56 & 160.80 \\
				& & & 35 & 101.43 & 96.1 & 51.8 & 1089.94 & 156.67 & 164.95 \\
				& & & 40 & 102.49 & 95.8 & 57.6 & 1061.11 & 157.64 & 163.67 \\
				\bottomrule[2pt]
			\end{tabularx}
			\begin{tablenotes}
				\item \scriptsize Notes:
				\item \scriptsize 1. BD: distribution of bidding, $N_{AB}$: the total number of vehicles, $N_B$: the number of Type-B vehicles.
				\item \scriptsize 2. AWT: average waiting time, ATR: average trip matching rate, ACR: average completion rate, ASW: average social welfare.
				\item \scriptsize 3. AP-A: average payoff of Type-A drivers, AP-B: average payoff of Type-B drivers.
			\end{tablenotes}
		\end{threeparttable}
		\label{tab:HTD_data_result}
	\end{table}
	
	\section{Discussions and Conclusion}\label{section:conc}
	This study presents an operational strategy for mobile crowd-sensing using ride-hailing taxis. The strategy tailors task assignment rules for the two vehicle types in the ride-hailing system, prioritizing rider service levels. We address the order assignment problem by matching riders to Type-A and Type-B drivers, employing an auction-based mechanism to assign MSTs to `redundant' Type-B drivers who have been idle for a significant duration. To ensure the data user's benefit and incentivize ride-hailing platforms to engage in mobile sensing, we enhance the VCG mechanism by introducing a budget balance rule for winner selection. Additionally, we devise a novel payment rule to maintain platform budget balance. The RBC-MST mechanism, satisfying IC, IR, and BB properties, is developed as favorable for all ride-hailing system stakeholders. Various scenarios are tested to evaluate the integrated operational strategy and the RBC-MST mechanism. Results indicate that the proposed strategy achieves substantial social surplus while maintaining a satisfactory completion rate.
	
	Another noteworthy observation is the benefit the RBC-MST mechanism provides to Type-B drivers who are inclined to undertake sensing tasks in most scenarios. An exception arises when the unit distance price of bids is constrained within the range of $[2,4]$, coupled with high demand for ride-hailing in the network. In this situation, Type-B drivers may only marginally increase their average earnings from completing MSTs compared to serving riders. However, they miss out on the opportunity to cater to more riders in busy travel hubs. This issue could be addressed by introducing more Type-B vehicles into the ride-hailing system or simply expanding the bounds for the unit price of valid bids.
	
	
	Although the proposed operation strategy and MST assignment mechanism effectively achieve our research objectives, there's room to tailor the methodology for more realistic scenarios. One direct extension could involve allowing Type-B drivers to handle both a ride-hailing order and one or multiple MSTs simultaneously. With appropriate route planning, a Type-B driver could efficiently complete MSTs during a detour while transporting a rider to their destination. This arrangement could potentially increase Type-B drivers' average income and enhance their willingness to undertake MSTs. Fortunately, implementing this concept is not difficult by introducing a tri-partite shareable network of drivers, MSTs, and riders, as described in prior research \citep{alonso2017,ge2021}. Ideal drivers can be selected by solving a similar matching problem as the one presented in Section \ref{section:rbc} on this shareable network.
	
	Another issue worthy of investigation is budget allocation across plan cycles. While this paper discusses two rules--the \textit{one-shot} rule and the \textit{proportionality} rule--both are greedy and may not ensure optimal budget allocation. Alternatively, the budget for each plan cycle could be optimized by \textit{learning} multi-day mobility and service patterns of drivers.
	
	\section*{CRediT authorship contribution statement}
	\textbf{Shenglin Liu}: Conceptualization, Formal analysis, Methodology, Data curation, Validation,Writing - original draft. \textbf{Qian Ge}: Conceptualization, Formal analysis, Methodology, Writing- original draft, Funding acquisition. \textbf{Ke Han}: Methodology, Writing - original draft, Supervision, Project administration, Funding acquisition. \textbf{Daisuke Fukuda}: Validation, Writing- original draft, Writing - review \& editing. \textbf{Takao Dantsuji}: Validation, Writing - original draft, Writing - review \& editing
	\section*{Declaration of competing interest}
	The authors declare that they have no known competing financial interests or personal relationships that could have appeared to influence the work reported in this paper.
	\section*{Data availability}
	The dataset and codes to reproduce the numerical experiments are available on request.
	\section*{Acknowledgment}
	This work is supported by the National Natural Science Foundation of China through grants no. 72101215 and 72071163, the Natural Science Foundation of Sichuan Province through grant no. 2022NSFSC1906, and Technology Innovation and Development Project of Chengdu Science and Technology Bureau through grant no. 2022-YF05-00839-SN.
	
	\bibliographystyle{plainnat}

\end{document}